\documentclass[12pt]{amsart}

 \usepackage{graphicx}
 \usepackage{amscd, color}

\usepackage{marginnote}

\swapnumbers
 \def\al{\alpha}
 \def\be{\beta}
 \def\de{\delta}
 
 \def\eps{\varepsilon}

 \def\ga{\gamma}

 \def\om{\omega}

 \def\VV{{\mathbf V}}
 
 \def\EN{{\mathcal E}}

  \def \UC {\mathrm{UC}}
\def \CC {\mathrm{C}}

 \def\R{{\mathbb R}}

 \def\N{{\mathbb N}}
 
  \def\A{{\mathcal  A}}

 \def\ov{\overline}

 \def\F{{\mathcal F}}

 \def\Lip{\mathrm {Lip \, }}

\def\txt{\quad\hbox}
\def \Txt{\qquad\hbox}

\newcommand{\OO}{\hbox{\rm O}}

  \renewcommand{\proofname}{{\bf Proof:}}

 \theoremstyle{plain}

  \swapnumbers

 \newtheorem{Thm}{Theorem}[section]
 
 \newtheorem{Lemma}[Thm]{\bf Lemma}
 
 \newtheorem{Corollary}[Thm]{\bf Corollary}
 \newtheorem{Theorem}[Thm]{\bf Theorem}
 \newtheorem{Proposition}[Thm]{\bf Proposition}

 \theoremstyle{definition}
 \newtheorem{Definition}[Thm]{\bf Definition}

 \theoremstyle{remark}
 \newtheorem{Remark}[Thm]{\bf Remark}

 \newtheoremstyle{Cl}% name
  {5pt}%      Space above
  {3pt}%      Space below
  {\sl}%   Body font
  {}%         Indent amount (empty = no indent, \parindent = para indent)
  {\it}% Thm head font
  {:}%        Punctuation after thm head
  {.5em}%     Space after thm head: " " = normal interword space;
  {}%         Thm head spec (can be left empty, meaning `normal')

 \theoremstyle{Cl}

 \def\begincproof{
                  \renewcommand{\proofname}{\it Proof:}
                  \begin{proof}
                 }

 \def\endcproof{
                \renewcommand{\qedsymbol}{$\diamondsuit$}
                \end{proof}
                \renewcommand{\qedsymbol}{\openbox}
                \renewcommand{\proofname}{\bf Proof:}
               }

 \renewcommand{\proofname}{{\bf Proof:}}

\title[Time--dependent equations on networks]
{Time--dependent Hamilton--Jacobi equations on networks}

\author{Antonio Siconolfi}
\address{Dipartimento di Matematica, Sapienza Universit\`a  di Roma, Italy.}
\email{siconolfi@mat.uniroma1.it}


\begin{document}
\maketitle

\begin{abstract} We study well posedness of time--dependent Hamilton--Jacobi equations on a network, coupled with a continuous initial datum and a flux limiter. We show existence and uniqueness of solutions as well as stability properties. The novelty of our approach is that comparison results are proved linking the equation to a suitable semidiscrete problem, bypassing doubling variable method. Further, we do not need special test functions, and perform tests relative to the equations on different arcs separately.
\end{abstract}

\section{Introduction}

True to the title, the purpose  of this  paper   is to study the well posedness of a time--dependent Hamilton--Jacobi equation,  coupled with suitable additional conditions, posed on a network.

 We consider a connected network $\Gamma$ embedded in $\R^N$ with a finite number of arcs $\ga$, which are regular simple curves parametrized in $[0,1]$,  linking  points of $\R^N$ called vertices, which make   up a set we denote by $\VV$.  We define a Hamiltonian on $\Gamma$ as a collection of Hamiltonians $H_\ga:[0,1] \times \R \to \R$, indexed by arcs, with  the crucial feature  that Hamiltonians associated to arcs possessing  different support, are totally unrelated.

The equations we deal with are accordingly of the form
\begin{equation}\label{intro1}
    u_t + H_\ga(s,u')=0  \Txt{in $(0,1) \times (0,+\infty)$}
\end{equation}
on each arc $\ga$, the aim  being  to uniquely select  distinguished viscosity type solutions  of each equation which can be assembled  together continuously, making up  a continuous function $u: \Gamma\times (0,+\infty) \longrightarrow \R$
with $u(\ga(s),t)$  solution of \eqref{intro1} for each $\ga$. To accomplish it, one has  to appropriately exploit the network geometry, via the adjacency condition between arcs and vertices, and the decisive issue for that is the right definition of supersolution. The subtle point in fact  is that the
conditions  for supersolutions are not the same  at all vertices,  but are given taking into account the
network structure, as made precise in   Definition \ref{defsupersol} {\bf (ii)}.

The problem  becomes discontinuous across all the  one--dimensional interfaces of the  form
\[\{(x,t), t \in [0,+\infty)\} \Txt{with $x \in \VV$,}\]
in contrast to what happens for the stationary version of this kind of equations, where the discontinuities are located at the vertices, that is to say: they are   finite and of zero  dimension.  This dimensional change explains why    the analysis  of evolutive equations on networks is by far more challenging than the stationary ones.

There are consequently few results available in the literature. The basic reference paper  is \cite{IM} by Imbert and Monneau, where the topic is treated  through PDE techniques, adapting tools from viscosity solutions theory, under the assumptions that the Hamiltonians in play are continuous, semiconvex and coercive. See also \cite{BI}, \cite{GIM} for applications of this theory.  A previous contribution of the same authors, with an additional coauthor, see \cite{IMZ}, requires  instead the Hamiltonians to be convex, and  attacks  the problem using control theoretic representation formulae.

 Here we prove existence, uniqueness and stability of solutions on the network   assuming convexity of the Hamiltonians, but without the growth conditions which allow applying Fenchel transform, so that  an action functional cannot be defined. In addition, the Hamiltonians we consider cannot be put in relation to any control model. In conclusion, though the Hamiltonians are convex, we do not have representation formulae for solutions at hand, and our techniques employ purely PDE methods.

One of the main discoveries in \cite{IM} is that to get well posedness of the evolutive problem,  the assignment of an initial datum at $t=0$  is not enough. It must actually be coupled with a condition regarding the time derivative of solutions on the discontinuity interfaces. They qualify as
 {\em flux--limited} the corresponding solutions. We adopt here the same point of view, and the terminology of {\em  flux limiter} as well.

We make in Remark \ref{compara} a comparison between our definition of solution and the one of \cite{IM}. They are clearly  the same outside the discontinuity interfaces, namely classical viscosity solutions. On the  interfaces, the definition of subsolution coincides as well, while regarding  supersolution, which is the most delicate  point, the formulation is different, and our definition is stronger. We believe that our pattern  is  more related to the geometrical sense of the definition, and   is more simple to write down,  in particular because  we take into account, for any arc,  also the arc with the opposite orientation.

In \cite{IM}, the method is first developed in the context of junctions, namely networks with a single vertex,  for Hamiltonians only depending on the momentum variable. It is  then generalized to Hamiltonians also depending on state variable and time, and defined on general networks. In our opinion this last part, which contains  interesting ideas, would deserve to be developed more.

 We do not need the preliminary step of junctions. We directly work, with Hamiltonian depending on state variable and momentum,  on a compact network, namely such that any arc has bounded length, with a general geometry and the unique limitation that no loops are admitted, namely  arcs for which initial and final point coincide.  We believe that our approach can also include the presence of loops, but this should require nontrivial adjustments. In \cite{IM}, unbounded arcs are admitted, but no loops.

 The approach of \cite{IM} is based on the construction of special test functions at the vertices, and a clever  adaptation of Crandall--Lions  doubling variable method to get the comparison result. Perron--Ishii method is used to prove existence of solutions.

 Our method is  different.  First of all, we do not use doubling variable techniques, but instead we prove  a comparison principle by  associating the Hamilton--Jacobi equation to a semidiscrete problem posed on the discontinuity interfaces.  This is the same road walked in \cite{PoSi}, \cite{SiconolfiSorrentino} for the stationary  case, even if the evolutive setting brings in some complications. The proof of  the comparison result for the semidiscrete problem turns out to be quite simple, and it is then transferred to the initial equation  exploiting the fundamental property that a continuous function $u:\Gamma \times [0,+\infty)   \to \R$  is solution of the main problem if and only $u(\ga(s),t)$ solves \eqref{intro1} in the viscosity sense  for any $\gamma$, and its trace on the discontinuity interfaces is solution of the semidiscrete problem.

A further relevant peculiarity of our techniques with respect to those in \cite{IM}, is that we do not use special test functions at the vertices, more generally, we do not need functions testing at the same time solutions of equations with different Hamiltonians. For our definition, it is enough to consider viscosity test functions for the equations \eqref{intro1}, separately considered, plus test functions on the discontinuity interfaces.
Finally,  we do not use Perron--Ishii method to  prove existence of solutions, but rely on  a more constructive technique, showing first existence for small time interval and then gluing  together the local solutions to get a solution global in time.

All in all, the main outputs of the paper are:
\begin{itemize}
  \item[--] comparison principle for uniformly continuous sub and supersolutions, see Theorem \ref{unicumque};
  \item[--] existence  results for Lipschitz continuous initial data, see Theorem \ref{faber}; and continuous initial data, see Proposition \ref{nichil};
  \item[--] existence of unique continuous solution, which is the maximal among the continuous subsolutions, it is in addition uniformly continuous if the initial datum is continuous, and Lipschitz continuous for Lipschitz continuous initial data, see Theorem \ref{unicumquebis} and Proposition \ref{orte}
  \item[--]  stability results, see  Corollary \ref{uniqlo}   and Theorem \ref{tombola}.
\end{itemize}

The paper is organized as follows:  In Section \ref{pre} we provide some preliminaries, we give the assumptions on the Hamiltonians, the definition of flux limiter and of solution in Section \ref{defsol}. The Section \ref{AAA} is devoted to the definition and the study of the main properties of an operator, denoted by $G$, which enters, together with the operator $F_x$ defined in Section \ref{semidiscrete},  in the definition of the semidiscrete equation. Roughly speaking, the operator $G$ allows taking into account the constraint due to the presence of the flux limiter. The proofs in this section are elementary.

In Section \ref{interval} we gather  material  regarding one--dimensional evolutive Hamilton--Jacobi equation on an interval. The focus is on (sub)solutions less than or equal to a given datum on part of the parabolic boundary. Since this point of view is quite unusual, we did not find the needed statements  in the literature. Therefore we have chosen  to prove everything in full details to make the presentation self--contained. This part is quite lengthy, but the arguments are traditional and plain. The corresponding proofs are mainly in Appendix \ref{A}.

The core  and the most innovative part of the paper is assembled in Sections \ref{semidiscrete} and \ref{well}. In Section \ref{semidiscrete}, we  write down the semidiscrete problem, prove a comparison principle for it, see Theorem \ref{unicum},  and study the connection with  the time--dependent Hamilton--Jacobi equation. The part regarding supersolutions is the most demanding one, see in particular Proposition \ref{fundasuper}. In Section \ref{well} we set down our main results, all the proofs are rather simple except that of the existence Theorem \ref{faber}.

Finally in Appendix \ref{AA}, we record some well--known results on $t$--partial sup convolutions we need in the paper.

\medskip

\subsection*{Acknowledgments}
The  author wishes to express his gratitude to the  Mathematical Sciences Research Institute in Berkeley (USA) for its kind hospitality in Fall 2018 during the trimester program ``{\it Hamiltonian systems, from topology to applications through analysis}'',
where this project was initiated.

\bigskip

\section{Preliminaries} \label{pre}

\subsection{Notations and basic definitions}

We denote by $\CC(\cdot)$, $\UC(\cdot)$, the spaces of real valued continuous and uniformly continuous functions, respectively. We further denote by $|\cdot|_\infty$ the uniform norm. Given $a$, $b$ in $\R$, we set
\[ a \vee b = \max \{a,b\}   \qquad  a \wedge b = \min \{a,b\}.\]
We set
 \[\R^+ = [0,+\infty),  \qquad  Q = (0,1) \times (0,+\infty)\]
Given an open rectangle $R=(a,b) \times (T_0,T_1) \subset Q$, with $T_1 \in \R \cup \{+\infty\}$,  we further  set
\begin{eqnarray*}
  \partial_p^+ R &=& [a,b] \times \{T_0\}  \cup \{b\} \times (T_0,T_1) \\
  \partial_p^- R &=& [a,b] \times \{T_0\} \cup \{a\} \times (T_0,T_1)\\
  \partial_p R &=&\partial_p^- R  \cup \partial_p^+ R
\end{eqnarray*}
For any $C^1$ function $\psi: Q \to \R$ and $(s_0,t_0) \in Q$, we denote by $\psi'(s_0,t_0)$ the {\em space} derivative, with respect to $s$,  at $(s_0,t_0)$, and by $\psi_t(s_0,t_0)$ or $\frac d{dt} \psi(s_0,t_0)$ the time derivative.  Given a Lipschitz continuous function, $w: Q \to \R$ we define
\[ \Lip w= \sup_{(s_1,t_1) \neq (s_2,t_2) \atop (s_i,t_i) \in Q} \frac{|w(s_1,t_1) -w(s_2,t_2)|}{|t_1-t_2| + |s_1-s_2|}.\]
The {\em (Clarke) generalized gradient} of a Lipschitz continuous function  $u: Q \to \R$ at $(s_0,t_0)$ is given by
\[  {\mathrm {\ov {co}}} \big \{(p,r)= \lim (u'(s_i,t_i), u_t(s_i,t_i)) \mid (s_i,t_i) \;\hbox{diffe. pts of $u$},\, (s_i,t_i) \to (s_0,t_0) \big \},\]
where $\mathrm{co}$ stands for convex hull, and is indicated by $\partial u(s_0,t_0)$.

Given a continuous function $u: \Gamma  \times \R^+ \to \R$ and an arc $\ga$ of $\Gamma$, we define $u \circ \ga:[0,1] \times [0,+ \infty)
\to \R$ as
 \[u\circ \ga(s,t)= u(\ga(s),t) \Txt{for any $(s,t) \in Q$.}\]
 Given a continuous function $u: Q \to \R$, we call {\em supertangents} (resp.  {\em subtangents}) to $u$ at $(s_0,t_0) \in Q$  the viscosity test functions from above (resp. below). If needed, we take, without explicitly mentioning, $u$ and test function coinciding at $(s_0,t_0)$ and test function strictly greater (resp. less) than $u$ in a punctured neighborhood of $(s_0,t_0)$.

 Given a closed subset $C \subset \ov Q$,  where $\ov Q$ stands for the closure of $Q$, we say that a supertangent  (resp. subtangent) $\varphi$ to $u$ at $(s_0,t_0) \in C$ is {\em constrained to $C$} if $(s_0,t_0)$ is maximizer (resp. a minimizer) of $u - \varphi$ in a neighborhood of $(s_0,t_0)$ intersected with $C$.

 The same notions apply, with obvious adaptations, to continuous function from $\R^+$ to $\R$.

\medskip
\subsection{Networks}  An {\it embedded network},  is a subset $
\Gamma \subset \R^N$ of the form
\[ \Gamma = \bigcup_{\ga \in \EN} \, \gamma([0,1]) \subset \R^N,\]
where $\EN$ is a finite collection of regular ({\it i.e.}, $C^1$
with non-vanishing derivative) simple oriented curves, called {\it
arcs} of the network,  that we assume, without any loss of
generality, parameterized on $[0,1]$.
 \\

Observe that on the support of any arc $\ga$, we  also consider the
inverse parametrization   defined as
\[- \ga(s)= \ga( 1 -s) \qquad\hbox{for $s \in [0,1]$.}\]
We call $- \ga$ the {\it inverse arc} of $ \ga$.  We assume
\begin{equation}\label{netw}
    \ga((0,1)) \cap \ga'([0,1]) = \emptyset \qquad\hbox{whenever $\ga \neq
\ga'$, $\ga \neq -{\ga'}$}.\\
\end{equation}

\medskip
 We call  {\it vertices} the
initial  and terminal points of the arcs, and denote  by  $\VV$ the
sets of all such vertices. Note that \eqref{netw} implies that
\[\ga((0,1)) \cap \VV  = \emptyset \qquad\hbox{for any $\ga \in
\EN$.}\] We assume that the network  is  {\em connected}, namely given two
vertices there is a finite concatenation of  arcs linking them. The unique  restriction we assume  on the geometry of the network is the nonexistence of loops, namely  arcs with initial and final point coinciding.   See \cite{Sunada} for a comprehensive treatment on graphs and networks.\\

Given $x \in \VV$, we define
\[ \Gamma_x =\{ \ga  \mid \ga(1)=x\}.\]

\medskip
A Hamiltonian on a network $\Gamma$ is a collection of Hamiltonians
$H_{\gamma}:[0,1] \times \R \to \R$, indexed by the arcs satisfying
\begin{equation}\label{ovgamma}
   H_{-\ga}(s,p) = H_{\ga}(1-s,-p) \qquad\hbox{for any $\ga
\in \EN$} \\
\end{equation}
Apart the above compatibility condition, the Hamiltonians $H_\ga$ are unrelated.

\bigskip
\section{Setting of the problem and definition of solution} \label{defsol}
We consider a Hamiltonian $\{H_\ga\}$ on the network $\Gamma$.
We  require  any  $H_\ga$ to be:
\begin{itemize}
    \item[{\bf(H1)}]  continuous in both arguments;
    \item[{\bf(H2)}]  convex   in the momentum variable;
    \item[{\bf(H3)}]  coercive  in the momentum variable, uniformly  in $s$;
    \item[{\bf (H4)}] uniformly local Lipschitz continuous  in $p$. Namely, given $M >0$, there exists $C_M$   such that
  \[ H_\ga(s,p) - H_\ga(s,q) \leq C_M \, |p-q| \Txt{for any $s \in [0,1]$, $q$, $p$ in $(-M,M)$}\]
   \end{itemize}

\smallskip

We set
\[c_\ga= - \max_s \, \min_p  H_\ga(s,p) \Txt{for any arc $\ga$.}\]
Note that the stationary equation
\[H_\ga(s,u')=a\]
admits (viscosity) subsolutions in $(0,1)$ if and only if $a \geq - c_\ga$.

Following \cite{IM}, we call {\em flux limiter} any   function  $x \mapsto c_x$ from $\VV$ to $\R$ satisfying
\[c_x  \leq \min_{\ga\in \Gamma_x} c_\ga \Txt{for any $x \in \VV$.}\]

\smallskip

Let $ (T_0,T_1)$ be an interval, possibly unbounded, contained in $\R^+$. For any given arc $\ga$, we consider  the time--dependent
equation
\begin{equation}\label{HJg}  \tag{HJ$_\ga$}
    u_t + H_\ga(s,u')=0 \Txt{in $ (0,1) \times (T_0,T_1)$.}
\end{equation}
We are interested in finding  a function $v: \Gamma \times [T_0,T_1) \to \R$  such that $v \circ \ga$ solves \eqref{HJg} in $R$, for any $\ga$, taking into account, in the sense we are going to specify, a flux limiter $c_x$ at any vertex. We denote by (HJ$\Gamma$) the problem as a whole.
\smallskip

The definition of (sub / super) solution to (HJ$\Gamma$) is as follows:

\smallskip

\begin{Definition}\label{defsupersol} We say that a  continuous  function $v(x,t)$,  $v: \Gamma  \times [T_0,T_1) \to \R$, is a {\em supersolution}  in $(T_0,T_1)$ if
\begin{itemize}
   \item[{\bf (i)}] $v\circ \ga$ is a viscosity supersolution of \eqref{HJg} in  $(0,1) \times (T_0,T_1)$ for any arc $\ga$;
     \item[{\bf (ii)}] for any vertex $x$ and  time $t_0 \in (T_0,T_1)$,  if
     \[ \frac d{dt} \phi (t_0) < c_x\]
     for some $C^1$ subtangent $\phi$ to $v(x,\cdot)$ at $t_0$, then
there is  an arc $\ga \in \Gamma_x$
such that  all the $C^1$  subtangents   $\varphi$,
constrained  to $[0,1] \times \ov {(T_0,T_1)}$, to
$v \circ \ga$  at  $(1,t_0)$   satisfy
     \[   \varphi_t(1,t_0) +  H_\ga(1, \varphi'(1,t_0))  \geq 0.\]
\end{itemize}
\end{Definition}
  Note  that  the arc $\ga$,  with $\ga(1)=x$, where   condition {\bf (ii)} holds true  changes in function of the time.

\smallskip

\begin{Definition}\label{defsubsol} We say that a  continuous  function $v(x,t)$, $v: \Gamma \times [T_0,T_1)  \to \R$, is a {\em subsolution} to (HJ$\Gamma$) in $(T_0,T_1)$ if
\begin{itemize}
   \item[{\bf (i)}] $v\circ \ga$ is a viscosity subsolution  of \eqref{HJg} in  $(0,1) \times (T_0,T_1)$ for any arc $\ga$;
     \item[{\bf (ii)}] for any vertex $x$ and  time $t_0 \in (T_0,T_1)$,  all supertangents $\psi(t)$ to $v(x,\cdot)$ at $t_0$ satisfy
     \[  \frac d{dt} \psi(t_0) \leq c_x.\]
\end{itemize}
\end{Definition}

We finally say that a  continuous function $v$ is {\em solution}  to (HJ$\Gamma$) in $(T_0,T_1)$  if it subsolution and supersolution at the same time.

\smallskip

\begin{Remark}\label{stufa} Given  a constant $a \in \R$ and  the family of Hamiltonians
\[H'_\ga(s,p)= H_\ga(s,p) + a \Txt{for any $\ga$,}\]
it is apparent that $c_x -a$ is a flux limiter for the $H'_\ga$. It is also apparent that $u$ is solution to (HJ$\Gamma$) with $H'_\ga$ in place of $H_\ga$, initial datum $g$ and flux limiter $c_x - a$, if and only if $u+a \,t$ is solution of the original problem (HJ$\Gamma$).  This means that the analysis of the problem is not affected if we add to all Hamiltonians the same constant. We can therefore assume, without any loss of generality, all the Hamiltonians $H_\ga$ to be strictly positive.
 This implies that flux limiter is   negative  at any vertex, and consequently  all subsolutions to \eqref{HJg} are decreasing in time.
\end{Remark}

\medskip

\begin{Remark}\label{compara}

We make some  comparisons  between our definition of solution and the one in \cite{IM}. Clearly, the point is to look at  the conditions  required on the interfaces
\[\{(x,t) \mid t \in [T_0,T_1)\}    \Txt{with $x \in \VV$}.\]
According to \cite[Theorem 2.10]{IM}, the notion of subsolution is the same.
As first pointed out in \cite{SC}, the definition of supersolution, for equations posed in networks,  is more delicate.

Our definition reads roughly like: if at a vertex $x$, for  some instant of time, the constraint given by the flux limiter is non active, then at least  for the equation on one arc $\ga$ ending at $x$,  some   state  constraint conditions must be satisfied.
This follows along the same  line of the definition given in \cite{PoSi}, \cite{SiconolfiSorrentino} for stationary equations.  In this case, due to the absence of the time variable, there is no flux limiter, so that just the state constraint condition survives at the vertices.

As far as we can see, our definition  is stronger in two respects.  First, we have more test functions from below at points on the interfaces  because we {\em perform tests separately}  on any branch  joining at the vertex under exam. On the contrary, in \cite{IM}  $C^1$ functions {\em testing from below at the same time all the equations} on the branches are used. Secondly, we require the supersolution property to be satisfied for {\em all} the test functions relative to a {\em distinguished equation}. On the contrary, in \cite{IM}, given a joint test function,   the validity of the supersolution property is just assumed  for {\em some} equation, which should {\em change  together with the test function}.
\end{Remark}

\bigskip

\section{The operator   $G$} \label{AAA}

Loosely speaking, the operator $G$ allows taking into account the constraint given, in problem (HJ$\Gamma$), by the flux limiter.  Given a time interval $[T_0,T_1) \subset \R^+$, with $T_1 \leq + \infty$, we  define $G:  \CC([T_0,T_1) \times (-\infty,0)) \to \CC([T_0,T_1))$ via
\begin{equation}\label{defG}
G[\psi,a](t)= \min \{  \psi(r)  +a  \,(t-r) ,\, r  \in [T_0,t].\}
\end{equation}

\smallskip

We recall two basic results that we will exploit in this section.

\smallskip
\begin{Lemma}\label{decresce} Let $u$ be a continuous function in an interval $[\al,\be]$ satisfying
\[ \frac d{dt} \varphi(t) \leq  0 \;\hbox{(resp. $\geq 0$)}  \]
for any $t \in (\al,\be)$, any $C^1$ supertangent  to $u$ at $t$. Then $u$ is nonincreasing (resp. nondecreasing) in $[\al,\be]$.
\end{Lemma}
\begin{proof}
 We treat the case $\leq 0$. Given $\eps >0$, the function $u_\eps:=u(t)- \eps \, t$ satisfies the assumptions with strict inequality. This implies that it cannot have local maximizers in $(\al,\be)$ and consequently it has at most one local minimizer. If it is not strictly  decreasing in $[\al,\be]$, there is a nontrivial subinterval, say $(\ga,\de)$,  where it is nondecreasing. Since the set of points where $u_\eps$ admits $C^1$ supertangent is dense in $[\al,\ga]$,  we find $t_0 \in (\ga,\de)$ and $\varphi$ supertangent to $u_\eps$ at $t_0$. Summing up: there is a neighborhood of $t_0$ where $\varphi$ is strictly decreasing, $u_\eps$ nondecreasing and $\varphi$ supertangent to $u_\eps$ at $t_0$. This is clearly impossible.

We derive that $u_\eps$ is decreasing in $[\al,\be]$ and, passing at the limit as $\eps \to 0$, we find that $u$ is nonincreasing, as was claimed. The case $ \geq 0$ can be proved arguing similarly.

\end{proof}

\smallskip

The same statement of above also holds by replacing supertangents by subtangents.

\begin{Lemma}\label{decrescebis} Let $u$ be a continuous function in an interval $[\al, \be]$ satisfying
\[ \frac d{dt} \varphi(t) \leq  0 \;\hbox{(resp. $\geq 0$)}  \]
for any $t \in (\al,\be)$ any $C^1$ subtangent  to $u$ at $t$. Then $u$ is nonincreasing (resp. nondecreasing) in $[\al,\be]$.
\end{Lemma}

\smallskip

\begin{Lemma} We have

\begin{equation}\label{pfb2tris}
 G[\psi,a](t) = \min \{ G[\psi,a](r)  +a  \,(t-r) ,\, r  \in [T_0,t]\} \Txt{for any $t \in [T_0,T_1)$.}
 \end{equation}
 \end{Lemma}
\begin{proof} We  set $w(t)= G[\psi,a](t)$.
It is apparent  that $w(t)$ is greater than or equal to the function in the right hand--side of \eqref{pfb2tris}.  Given $ r \leq t$, we find for a suitable  $r' \leq r$
\[w(r) + a \, (t-r)= \psi(r') +a \, (r-r') +a \, (t -r) = \psi(r') +a \, ( t-r') \geq w(t),\]
which  proves
\[ w(t) \leq \min \{  w(r)  +a  \,(t-r) ,\, r  \in [T_0,t]\} \Txt{for any $t \in [T_0,T_1)$}\]
and concludes the proof.
\end{proof}

\smallskip
\begin{Lemma}\label{proprovvi}  $G[\psi,a]$ is the maximal  continuous function in $[T_0,T_1)$ less than or equal to $\psi$ satisfying
\begin{equation}\label{uffa}
  \frac d{dt}\varphi (t) \leq a \txt{for any $t \in (T_0,T_1)$, any  $C^1$ supertangent $\varphi$  to $G[\psi,a]$ at $t$ .}
\end{equation}
\end{Lemma}
\begin{proof} We  set  $w(t)= G[\psi,a](t)$.    We deduce  from \eqref{pfb2tris} that $w$ is decreasing. Given $T> T_0$; we denote by $\om$ a continuity modulus of $w$ in $[T_0,T]$. We consider  $r$,  $t$  in $[T_0,T]$ with $r < t$, and denote by  $r_0 \in [T_0, t]$ a time realizing the equality in \eqref{defG}. If $r_0 \leq r$ then
\begin{eqnarray*}
  |w(r)- w(t)| &=& w(r)-w(t) \leq \psi(r_0) + a \, (r-r_0) - \psi(r_0) -a\,(t-r_0) \\
   &=& - a \, |t-r|.
\end{eqnarray*}
If instead $r_0 > r$ we obtain
\begin{eqnarray*}
  |w(r)- w(t)|  &\leq& \psi(r)  - \psi(r_0) - a \,(t-r_0) \leq \om (r_0-r) -a  \, (t-r_0) \\
 &\leq& \om(|t-r|)- a\, |t-r|.
\end{eqnarray*}
The above formulae show that  $w$ is continuous. If $\varphi$  is  a $C^1$ differentiable supertangent to $w$ at $t$, we have by \eqref{pfb2tris}
\[- \frac d{dt} \varphi (t) = \lim_{h \to 0^+} \frac {\varphi(t-h) -\varphi(t)}h \geq \lim_{h \to 0^+} \frac {w(t-h) -w(t-h)-a \, h}h = -a.\]
 Finally, If $v \leq \psi$ is a continuous function in $[T_0,T_1)$ satisfying  \eqref{uffa}  with $v$ in place of $G[\psi,a]$, then, given $t \in [T_0,T_1)$, we have by Lemma \ref{decresce}
\[v(t) \leq v(r)  + a \, (t-r)  \leq \psi(r)  + a \, (t-r) \txt{for any $r \in [T_0,t)$}.\]
This implies
\[v(t) \leq w(t).\]
and concludes the proof.

\end{proof}

\smallskip

\begin{Remark} \label{cazzopapa}
 We deduce from the proof of the above lemma that if $\psi$ is uniformly continuous  with continuity modulus $\om$ in  $[T_0,T_1)$,  then $G[\psi,a]$ is uniformly continuous as well with continuity modulus
 \[r \mapsto \om(r) - a \, r.\]
In addition, if $\psi$ is   Lipschitz continuous, then $G[\psi,a]$ is Lipschitz continuous as well,  it is maximal in the family of Lipschitz continuous functions $w: [T_0,T_1) \to \R$   satisfying
\[ v(t) \leq \psi(t) \txt{and} \qquad \frac d{dt} v(t) \leq a \txt{for a.e. $t$,}\]
and  has Lipschitz constant $-a \vee \Lip \psi$.

\end{Remark}
\smallskip
We record for later use:

\begin{Lemma}\label{lippi}
Assume that $\psi_n$ is a sequence of continuous functions uniformly converging to a function $\psi$ in $[T_0,T_1)$, then
\[G[\psi_n,a] \to G[\psi,a] \Txt{uniformly, in $[T_0,T_1)$.}\]
\end{Lemma}
\begin{proof}
Given $\eps >0$, we have for $n$ large
\[G[\psi_n,a]+ \eps= G[\psi_n+\eps,a] \geq G[\psi,a] \geq G[\psi_n -\eps,a] = G[\psi_n,a]- \eps.\]
\end{proof}

\smallskip

\begin{Lemma}\label{provvi} Assume $\psi_1$, $\psi_2$ to be  continuous functions from $[T_0,T_1)$ to $\R$  satisfying
\[ \psi_1 > \psi_2 \Txt{in $[T_0,T]$, for some $T > T_0$}\]
then
\[G[\psi_1,a] > G[\psi_2,a] \Txt{in $[T_0,T]$, for any $a <0$.}\]
\end{Lemma}
\begin{proof} Because of the continuity of $\psi_1$, $\psi_2$ there exists $b >0$ with
\[ \psi_1 > \psi_2+ b   \Txt{in $[T_0,T]$}\]
therefore
\[ G[\psi_1,a] \geq G[\psi_2+ b,a] = G[\psi_2,a] + b > G[\psi_2,a]\]
\end{proof}

\smallskip
The next result will be  used in Proposition \ref{fundasol}.

\begin{Proposition} \label{overpfb2}  Let   $(\psi,a) \in \CC([T_0,T_1))\times (-\infty,0)$. If  $G[\psi,a]$ admits a $C^1$ subtangent $\varphi$ at $t_0 \in (T_0,T_1)$  with $\frac d{dt} \varphi (t_0) < a$ then $G[\psi,a](t_0)= \psi(t_0)$.
\end{Proposition}
\begin{proof}
 We set $w = G[\psi,a]$. We take  $r_0$ realizing the equality in \eqref{pfb2tris} for  $t_0$ and assume $r_0 \in [T_0, t_0)$; for  $r \in [r_0,t_0]$, we have
 \begin{eqnarray*}
   w(t_0) &\leq&  w(r) + a \, (t_0-r) \leq w(r_0) +a  \, (r-r_0) +a  \, (t_0-r) \\
   &=& w(r_0)+ a  \, (t-r_0)
 \end{eqnarray*}
and since the first and last term in the above formula are equal, we conclude
\[ w(r)= w(r_0) + a  \, (r-r_0) \Txt{for $r \in [r_0,t_0]$}. \]
This is in contrast with the existence of a subtangent  $\varphi$ to $w$ at $t_0$  with $\frac d{dt}  \varphi(t_0)< a$.  We conclude  that $r_0=t_0$, which proves the assertion by  the very definition of $w$, see \eqref{defG}.

\end{proof}

\smallskip
We finally have:

\smallskip

\begin{Proposition}\label{ginetto} Let $w$ be a continuous function in $[T_0,T_1)$ and   $(\psi,a) \in \CC([T_0,T_1))\times (-\infty,0)$. Assume that $w(T_0) \geq \psi(T_0)$, and $w(t) \geq \psi(t)$ whenever there is a $C^1$ subtangent $\varphi$ to $w$ at $t$ with $\frac d{dt} \varphi(t) < a$.
Then
\[ w(t) \geq G[\psi,a](t)    \Txt{for any $t \in [T_0,T_1)$}\]
\end{Proposition}
\begin{proof}  Given $t \in (T_0,T_1)$, we define $E$ as the set of points  $r \in (T_0,t)$ where there is  a subtangent $\phi$ to $w$  at $r$ with  $\frac d{dt} \phi(r) <a$.
We set
\[r_0= \left \{ \begin{array}{cc}
          \sup E & \txt{if $E \neq \emptyset$} \\
          T_0 & \txt{if $E = \emptyset$}
        \end{array} \right . \]
 By the assumption and $w(T_0) \geq \psi(T_0)$,  we have $w(r_0)\geq \psi(r_0)$ and by Lemma \ref{decresce}
\[w(t) \geq  w(r_0) +a \, (t-r_0) \geq \psi(r_0) +a \, (t-r_0).\]
Therefore $w(t) \geq G[\psi,a](t)$. This concludes the proof.
\end{proof}

\smallskip

\bigskip

\section{Hamilton Jacobi equations in an interval.}\label{interval}
In this section we consider a single Hamiltonian $H: [0,1] \times \R \to \R$  satisfying   {\bf (H1)}, {\bf (H2)}, {\bf (H3)}, {\bf (H4)} plus
\begin{equation}\label{assu}
   \max_{s \in [0,1]} \, \min_{p \in \R} H(s,p)  > 0,
\end{equation}
 see Remark \ref{stufa}, we further consider the equation
\begin{equation}\label{HJloc}  \tag{HJ}
    u_t + H(s,u')=0.
\end{equation}

We fix an open  rectangle $R = (a,b) \times (T_0,T_1) \subset Q$, possibly unbounded.  We call {\em admissible} an uniformly continuous function $w_0$ defined on $\partial_p^-R$ (resp. $\partial_p^+R$, $\partial_p R$) if there is an uniformly continuous subsolution of \eqref{HJloc} in $R$ agreeing with $w_0$ on $\partial_p^-R$  (resp. $\partial_p^+R$, $\partial_p R$).\\

\medskip
\subsection{Basic facts}     We start recalling some well known results on \eqref{HJloc}.  the first one is a   comparison result  when the boundary datum is assigned on the whole of parabolic boundary.

\smallskip

\begin{Theorem}\label{barles} Let $u$, $v$ be continuous sub and supersolution, respectively, to \eqref{HJloc}, in $R$ with  $u$  Lipschitz continuous.  If $u \leq v$ on $\partial_p R$ then $u \leq v$ in $R$.
\end{Theorem}

\smallskip

This result can be generalized using $t$-- partial sup--convolutions, see Appendix \ref{AA}.

\begin{Theorem}\label{nobarles} Let $u$, $v$ be continuous sub and supersolution, respectively, to \eqref{HJloc}, in $R$ with  $u$  uniformly continuous.  If $u \leq v$ on $\partial_p R$ then $u \leq v$ in $R$.
\end{Theorem}
The proof is in Appendix \ref{A}.

\smallskip

The next proposition says that the vertical (in time) gluing of two (sub/super) solutions is still a (sub/super) solution.

\smallskip

\begin{Proposition} \label{melaton} Let $t^* \in (T_0,T_1)$.   Let $u$ be a  continuous function from $R$ to $\R$. Assume $u$ to be (sub/ super)solution of \eqref{HJloc} in $(a,b) \times (T_0,t^*) $ and in $(a,b) \times (t^*,T_1)$. Then $u$ is (sub/super)solution in $R$.
 \end{Proposition}

 \smallskip

 We proceed stating  a result on maximal subsolutions.
\smallskip

\begin{Proposition} \label{2021} Let $w_0$ be a   continuous  admissible datum  on  $\partial^+_p R$ (resp. $\partial_p^- R$, $[a,b] \times \{T_0\}$). The maximal subsolution of \eqref{HJloc} attaining the datum $w_0$ on $\partial_p^+ R$ (resp. $\partial_p^- R$, $[a,b] \times \{T_0\}$), denoted by $w$, is characterized  by the properties of being solution in $R$, and to satisfy for any $(s^*,t^*) \in \partial_p R \setminus \partial^+_p R$ (resp. $\partial_p R \setminus \partial^-_p R$, $\partial_p R \setminus \big ( [a,b] \times \{T_0\} \big )$ any   subtangent $\varphi$,  constrained  to $\ov R$, to $w$  at $(s^*,t^*)$
\[\varphi_t(s^*,t^*) + H(1,\varphi'(s^*,t^*)) \geq 0.\]
\end{Proposition}

\smallskip

We derive the following stability property:

\begin{Corollary} \label{2021bis} Let $H_n$ be a sequence of Hamiltonians in $\ov R$ satisfying the same assumptions of $H$ and locally uniformly convergent  to $H$, and $w^0_n$ a sequence of  continuous initial data in $\partial^-_p R$ locally uniformly convergent to $w_0$. Then the maximal subsolution $w_n$ of \eqref{HJloc}, with $H_n$ in place of $H$, attaining the datum $w^0_n$ on $\partial_p^-R$ locally uniformly converges in $\ov R$ to the maximal subsolution $w$ of \eqref{HJloc} agreeing with $w_0$ on $\partial_p^-R$.
\end{Corollary}

\smallskip

We finally record for later use:

\begin{Proposition}\label{hoje} Assume that $u$ is a Lipschitz continuous subsolution to \eqref{HJloc} in $R$. Assume further that
\begin{equation}\label{hoje1}
\Lip u(s,\cdot) \leq M \Txt{for some $M >0$, any $s \in [a,b]$}.
\end{equation}
Then $u (\cdot,t)$ is subsolution to
\begin{equation}\label{hoje2}
 H(s,v') \leq M \Txt{in $(a,b)$, for any $t \in \ov{[T_0,T_1)}$.}
\end{equation}
\end{Proposition}

The proof is in Appendix \ref{A}.

\medskip

\subsection{Maximal subsolutions}

\begin{Proposition}\label{maxistar}  Let $w_0$ be a Lipschitz  continuous    boundary  datum assigned on  $\partial_p^- R$. Then the  function
\begin{equation}\label{maxistar1}
  v(s,t) := \sup \{u(s,t) \mid u \;\hbox{un. cont. subsoln of \eqref{HJloc} in} \; R, u \leq w_0 \;\hbox{on $\partial_p^- R$}\}
\end{equation}
is a Lipschitz continuous  solution to \eqref{HJloc}  in $R$  with
\begin{eqnarray}
  |v(s,t_1) - v(s,t_2)| &\leq& M_0 \, |t_1-t_2|   \label{trump1bis}\\
  |v(s_1,t) - v(s_2,t)| &\leq& L_0 \, |s_1-s_2|, \label{trump2bis}
\end{eqnarray}
 where
\begin{eqnarray}
  M_0 &=& \min \{m \mid H(s,w'_0(\cdot,T_0)) \leq m \;\hbox{a.e. $s$}\} \vee \Lip w_0(a,\cdot)   \label{trump3}\\
  L_0 &=& \max\{|p| \mid H(s,p) \leq M_0 \, \forall s\}  \label{trump4}.
\end{eqnarray}
In addition, it  coincides with $w_0$ in $[a,b] \times \{T_0\}$.
\end{Proposition}

In other terms, $M_0$ is the minimal constant such that the initial datum is subsolution of the corresponding stationary equation, and $L_0$ is a constant estimating from above  the Lipschitz constants of all  subsolutions to such stationary equation.  The proof is in Appendix \ref{A}.

\medskip

\begin{Remark}\label{fabro}
If the boundary datum $w_0$ is   assigned on $[a,b] \times \{T_0\}$ and we define $v$ as in \eqref{maxistar1}, we get, slightly adapting the  argument of Proposition \ref{maxistar},  that  $v$ is Lipschitz continuous solution to \eqref{HJloc} agreeing with $w_0$ in $[a,b] \times \{T_0\}$  and satisfying the estimates \eqref{trump3}, \eqref{trump4} with $M_0$, $L_0$ replaced by
\begin{eqnarray}
  M &=& \min \{m \mid H(s,w'_0(\cdot,T_0)) \leq m \;\hbox{ a.e. $s$}\} \label{trump1} \\
  L &=& \max\{|p| \mid H(s,p) \leq M \; \forall s\}. \label{trump2}
\end{eqnarray}
If the datum $w_0$ is assigned on $\partial^+_p R$ the constants to put in \eqref{trump3}, \eqref{trump4} are
\begin{eqnarray}
  M_1 &=& \min \{m \mid H(s,w'_0(\cdot,T_0)) \leq m \;\hbox{a.e. $s$}\} \vee \Lip w_0(b,\cdot)  \label{trump5}\\
  L_1 &=& \max\{|p| \mid H(s,p) \leq M_1 \, \forall s\} \label{trump6}.
\end{eqnarray}
 Finally, if the datum $w_0$ is given on the whole of parabolic boundary $\partial_p R$, we have the constants
\begin{eqnarray}
  M_2 &=& \min \{m \mid H(s,w'_0(\cdot,T_0)) \leq m \;\hbox{a.e. $s$}\} \vee \Lip  w_0(a,\cdot) \vee \Lip w_0(b,\cdot) \label{trump7}\\
  L_2 &=& \max\{|p| \mid H(s,p) \leq M_2 \, \forall s\} \label{trump8}.
\end{eqnarray}
\end{Remark}

\smallskip

We generalize Proposition \ref{maxistar} to absolutely continuous boundary data.

\smallskip

\begin{Proposition} \label{maxistaruni} Let $w_0$ be an uniformly  continuous    boundary  datum assigned on  $\partial_p^- R$. Then the  function
\begin{equation}\label{maxistar111}
  v(s,t) := \sup \{u(s,t) \mid u \;\hbox{un. cont. subsoln of \eqref{HJloc} in} \; R, u \leq w_0 \;\hbox{on $\partial_p^- R$}\}
\end{equation}
is an  uniformly  continuous  solution to \eqref{HJloc}  in $R$, and
coincides with $w_0$ in $[a,b] \times \{T_0\}$.
\end{Proposition}

\smallskip

We preliminarily need introducing a regularization device. Given an uniformly continuous function $u$  defined in a closed set  $C \subset \ov Q$, we define, for $n \in \N$,  the following approximations from above and below
\begin{eqnarray}
  u^{[n]}(s,t))&=& \sup \{u(z,r) - n\, ( |z-s| + |t-r|) \mid (z,r) \in C\}  \label{hopfsup}\\
 u_{[n]}(s,t)&= & \inf \{u(z,r) + n\, ( |z-s| + |t-r|) \mid (z,r) \in C\} \label{hopfinf}
\end{eqnarray}
The following properties hold:

\begin{Lemma} \label{hopf} \hfill
\begin{itemize}
   \item [{\bf (i)}] For $n$ sufficiently large, the functions $u^{[n]}$ and $u_{[n]}$ are Lipschitz continuous in $C$ with Lipschitz constant $n$, and $u^{[n]} \geq u \geq u_{[n]}$;
   \item [{\bf (ii)}] $u^{[n]}$ and  $u_{[n]}$  uniformly converge to $u$ in $C$ as $n$ goes to infinity, with $|u^{[n]}-u|_\infty$,  $|u_{[n]}-u|_\infty$  only depending on the continuity modulus of $u$.
\end{itemize}
\end{Lemma}

The proof is in Appendix \ref{A}.

\smallskip

\begin{proof}[{\bf Proof of Proposition \ref{maxistaruni}}]  According to Lemma \ref{hopf}, we find  a sequence of Lipschitz continuous functions $w^0_n$ uniformly converging to $w_0$ in $\partial_p^-R$. We denote by $v_n$ the maximal subsolutions of \eqref{HJloc} not exceeding $w^0_n$ in $\partial_p R^-$. According to Proposition \ref{maxistar}, the $v_n$'s are Lipschitz continuous subsolutions of \eqref{HJloc}, and  any $v_n$ agrees with $w^0_n$ on $\partial_p^- R$. Given $\eps >0$, we have
\[ w^0_n - \eps \leq w_0 \leq w^0_n + \eps \Txt{for $n$ large, in $\partial_p^- R$}\]
which implies by the very definition of maximal subsolution
\[v_n - \eps \leq v \leq v_n + \eps \Txt{for $n$ large, in $R$}\]
so that $v_n$ uniformly converges to $v$ in $R$. This gives the assertion.
\end{proof}

\smallskip
We also derive from the argument of the above proposition:

\begin{Corollary}\label{cormaxistaruni} Let $w_0$ be an uniformly continuous function in $\partial_p^-R$, and $w^0_n$ a sequence of uniformly continuous functions uniformly approximating it in $\partial_p^-R$. Then the maximal subsolutions of \eqref{HJloc} among those less than or equal to $w^0_n$ in $\partial_p^-R$ uniformly converge  to  the maximal subsolutions less than or equal to $w_0$ in $\partial_p^-R$.

\end{Corollary}

\medskip
\subsection{Admissible traces}
In this section we investigate the possibility of modifying an admissible boundary datum on part of the parabolic boundary, still getting an admissible datum. The first statement in this respect is:

\begin{Proposition}\label{may} Let $u_0$ be a Lipschitz continuous  admissible trace on $\partial^-_p R$, and assume   $w_0$ to be a Lipschitz continuous function  defined in  $\partial^-_p R$ with
\[ u_0 \leq w_0 \;\;\hbox{in $[a,b] \times \{T_0\}$ and }  \;\; u_0  = w_0  \;\;\hbox{in $\{a\} \times [T_0,T_1)$}\]
then $w_0$ is admissible on $\partial^-_p R$ as well.
\end{Proposition}
\begin{proof}  We take $M$ with
\begin{eqnarray*}
  \Lip w_0, \, \Lip u_0 &<& M  \Txt{in $\partial^-_p R$}\\
 H(s, w_0'(\cdot,0))&<& M  \Txt{a.e. in $(a,b)$}
\end{eqnarray*}
Therefore   $\ov w(s,t): =  w_0(s,T_0) - M \, (t -T_0)$ is a  subsolution  to \eqref{HJloc}  in $R$ such that
\begin{eqnarray*}
  \ov w & \leq & u_0 =w_0 \quad\hbox{$\{a\} \times (T_0,T_1)$} \\
  \ov w &=& w_0 \geq u_0  \txt{in $[a,b] \times \{T_0\}$}
\end{eqnarray*}
We denote by  $\ov u$ a subsolution with trace $u_0$ on  $\partial^-_p R$, then the function
\[(s,t) \mapsto \max \{ \ov w(s,t), \, \ov u(s,t) \}\]
is a  subsolution agreeing with $w_0$ on $\partial^-_p R$, as was claimed.

\end{proof}

\smallskip

The following result  somehow complements Proposition \ref{may}. We show that we can  fix an admissible boundary datum on $\partial_p^-R$ (resp. $\partial_p^+R$) and make it decrease on $\{a\} \times [T_0,T_1)$ (resp. $\{b\} \times [T_0,T_1)$) still obtaining an admissible datum. We need for that  an additional assumption on the time derivative of the new datum on $\{a\} \times [T_0,T_1)$ (resp. (resp. $\{b\} \times [T_0,T_1)$)).\\

\smallskip
\begin{Proposition}\label{lemfaberuno}  Let $u_0$ be an admissible Lipschitz continuous  boundary
 datum for \eqref{HJloc} in $\partial^-_p R$ (resp. $\partial^+_p R$) ,
and $v_0$ a Lipschitz continuous function defined in  $\partial^-_p R$ (resp. $\partial^+_p R$) with $u_0 = v_0$ on
$[a,b] \times \{T_0\}$ and $v_0 \leq u_0$ on $\{a\} \times [T_0,T_1)$ (resp. on $\{b\} \times [T_0,T_1)$). We  further  assume   that
\begin{equation}\label{catania55}
 \frac d{dt} v_0 (a,t) \leq  - \max_{s \in [0,1]} \, \min_{p \in \R} H(s,p) \txt{a.e. in $(T_0,T_1)$.}
\end{equation}
\[ \left ( \frac d{dt} v_0 (b,t) \leq  - \max_{s \in [0,1]} \, \min_{p \in \R} H(s,p) \txt{a.e. in $(T_0,T_1)$.} \right )\]
 Then $v_0$ is admissible for \eqref{HJloc} on  $\partial^-_p R$ (resp. $\partial^+_p R$).
\end{Proposition}

We need a preliminary elementary  lemma, we provide the proof in the Appendix  \ref{A} for reader's convenience.

\smallskip

\begin{Lemma}\label{catania1}  The minimum  of two Lipschitz continuous subsolutions to \eqref{HJloc}  is a subsolution.
\end{Lemma}

\smallskip

\begin{proof} [{\bf Proof of Proposition \ref{lemfaberuno}:} ]  We give the proof for functions defined on $\partial_p^- R$. The argument for $\partial_p^+ R$ is the same.  We set
\[c= - \max_{s \in [0,1]} \, \min_{p \in \R} H(s,p).\]
  We denote  by $w_0(s)$ a function satisfying
\[H(s,w_0') = -c  \Txt{in the viscosity sense in $(a,b)$}\]
and $u_0(a)= 0$.  Due to \eqref{catania55}, the function
\[w(s,t):=w_0(s) + v_0(a,t)\]
is a Lipschitz continuous subsolution to \eqref{HJloc} in $R$.
Applying   Proposition \ref{may} with admissible trace $w$, we find that the trace
\begin{eqnarray}
  \max \{w, \, u_0\} =  \max \{w, \, v_0\} && \txt{on $[a,b] \times \{T_0\}$,} \label{lemfaberuno1}\\
 v_0(a,t) && \txt{on $\{a\} \times [T_0,T_1)$} \label{lemfaberuno11}
\end{eqnarray}
is admissible for \eqref{HJloc} on $\partial_p^-R$.    Finally,  by exploiting  Lemma \ref{catania1} we see that $v_0$, being  the  minimum  of $u_0$ and the function in \eqref{lemfaberuno1}, \eqref{lemfaberuno11}, is admissible.
\end{proof}

\smallskip

We say  that a continuous function $u$ is {\em strict} subsolution in $R$ of \eqref{HJloc}   if
\begin{equation}\label{stretto}
u_t + H(s,u') \leq - \de \qquad\hbox{ in $R$, in the viscosity sense }
\end{equation}
for some $\de >0$.
\smallskip

\begin{Lemma}\label{lemfabertre} Let $u$ be an uniformly  continuous strict subsolution to \eqref{HJloc} in $R$.
The maximal  subsolution  agreeing with  $u$  on $\partial^-_p R$ is
strictly greater than $u$ in $R \cup \big ( \{b\} \times (T_0, T_1 ) \big ) $.
\end{Lemma}
\begin{proof} We can assume that $u$ is the maximal uniformly continuous subsolution to \eqref{stretto} among those with the same trace on $\partial^-_p R$. We denote by $u_n$ a sequence of Lipschitz continuous functions uniformly approximating $u$ on $\partial_p^- R$ and by $\ov u_n$ the maximal subsolution of \eqref{stretto} less that or equal to $u_n$ on $\partial_p^- R$. We know by Corollary \ref{cormaxistaruni} that the $\bar u_n$'s  uniformly converge to $u$ on $\ov R$. We finally denote by $v$ the maximal subsolution of \eqref{HJloc} agreeing with $u$ on $\partial_p^-R$.

It is clear that $u \leq v$ on $\ov R$. if $u=v$ at a point  $(s_0,T_0) \in \ov R \setminus \partial_p^-R$, then, taking into account that $\ov u_n$ uniformly converges to $u$,  we find that $\ov u_n$ is subtangent to $v$ at some point $(s_n,t_n) \in \ov R \setminus  \partial_p^-R$, for $n$ large. Exploiting that $\ov u_n$ is subsolution to \eqref{stretto}, we can construct, using Perron--Ishii method, a subsolution to \eqref{HJloc} strictly greater than $v$ in a neighborhood of $(x_n,t_n)$ and equal to $u$ in $\partial_p^-R$. This is impossible by the very definition of $v$.
 \end{proof}

 \smallskip

 We record for later use in the proof of Proposition \ref{fundasuper}.

 \smallskip
 \begin{Proposition}\label{xanax}  Given
$s_0 \in (a,b)$, $c \leq - \max_s \, \min_p H(s,p)$,   we set
\[A = (a,s_0) \times (T_0,T_1) \Txt{and} \qquad  B = (s_0,b) \times (T_0, T_1). \]
 Let $u$ be an uniformly  continuous supersolution of \eqref{HJloc} in $R$,  we denote by  $v$ the maximal subsolution  in $A$ with
trace less than or equal to $u$  on $\partial^-_p A$, and  by $w$ the maximal subsolution in $B$ with trace less than or equal to $u$  on $\partial^+_p B$. Then
\begin{equation}\label{xanax00}
  G[u(s_0,\cdot),c](t) \geq \min \{G[v(s_0,\cdot),c],G[w(s_0,\cdot),c]\}(t) \Txt{for any $t \in (T_0,T_1)$.}
\end{equation}
We recall that the operator $G$ is defined in Section \ref{AAA}.
\end{Proposition}
\begin{proof}  We denote by $u_n$, with $u_n \leq u$,  the Lipschitz continuous  approximation from below of  $u$ in $\partial_p R$  introduced in Lemma \ref{hopf}. We further denote $\ov v_n$, $\ov w_n$ the maximal subsolutions on $A$, $B$, respectively, with trace less than or equal to $u_n$ on $\partial_p^-A$, $\partial_p^+ B$, respectively.  We know by Corollary \ref{cormaxistaruni} that
\begin{equation}\label{xanax1}
  \ov v_n \to v \txt{uniformly in $\ov A$} \txt{and} \quad  \ov w_n \to v \txt{uniformly in $\ov B$}\end{equation}
Owing to Proposition \ref{lemfaberuno}  the function equal to $u_n$ in $[a,s_0]  \times \{T_0\}$ and
\begin{equation}\label{xanax33}
 \min \{G[\ov v_n(s_0,\cdot),c],G[\ov w_n(s_0,\cdot),c]
\end{equation}
 on $\{s_0\} \times (T_0, T_1)$  is admissible on $\partial^+_p A$, and the same  holds true on $\partial^-_p B$ for the function equal to $u_n$ on
$[s_0,b]  \times  \{T_0\}$  and to the  function  in \eqref{xanax33}
 on $\{s_0\} \times (T_0, T_1)$.
We denote by $\widetilde v_n$, $\widetilde v_n$ the corresponding maximal subsolutions on $A$, $B$, respectively. We further set
\[ v^*_n = \min\{\ov v_n, \widetilde v_n\} \Txt{and} \qquad w^*_n = \min\{\ov w_n, \widetilde w_n\},\]
by Lemma \ref{catania1} $v^*_n$ and $w^*_n$ are subsolutions to \eqref{HJloc} in $A$ and $B$, respectively.
We have
\begin{eqnarray*}
  v^*_n &\leq& u_n \leq u \Txt{on $\partial_p^-R$} \\
  w^*_n &\leq& u_n \leq u \Txt{on $\partial_p^+R$} \\
  v^*_n &=& u_n \leq u \Txt{on $[a,s_0] \times \{T_0\}$} \\
  w^*_n &=& u_n \leq u \Txt{on $[s_0,b] \times \{T_0\}$}\\
  v^*_n &=&w^*_n \qquad\Txt{on $\{s_0\} \times (T_0,T_1)$}
\end{eqnarray*}
We consider the function  $\varphi_n$   defined in the whole of $R$  by merging together $ v^*_n$, $ w^*_n$. Note that $\varphi_n$ is Lipschitz continuous. In addition, since $\varphi_n$ is subsolution in $A$ and in $B$, it is subsolution in the whole of $R$ for the Hamiltonian is convex in $p$ so that  the notions of viscosity and a.e. subsolution coincide. We also have
\[ \varphi_n  \leq u \Txt{for any $n$, on $\partial R$.}\]
We therefore  get  by the comparison principle given in Theorem \ref{nobarles}
\[\varphi_n \leq u \Txt{ in $R$.}\]
 and consequently
 \[\varphi_n(s_0,\cdot)= \min \{G[\ov v_n(s_0,\cdot),c],G[\ov w_n(s_0,\cdot),c]\}(t) \leq G[u(s_0,\cdot),c](t)  \Txt{ for any $t$.} \]
We  get \eqref{xanax00} passing at the limit as $n$ goes to infinity in the above formula and taking into account \eqref{xanax1} plus Lemma \ref{lippi}.

\end{proof}

 \medskip

 \subsection{Finite speed of propagation }
In this section we assume the rectangle
$R = (a,b) \times (T_0,T_1)$ to be bounded with $a=0$, $b=1$.

\smallskip

  \begin{Lemma}\label{propaga1}  Let $u_0$ be  a Lipschitz continuous  initial datum on $[0,1] \times \{T_0\}$. We consider two  Lipschitz continuous solutions $u$, $v$ to \eqref{HJloc} in $R$ agreeing with  $u_0$ on $[0,1] \times \{T_0\}$. Then  there is $\de
>0$ depending on $H$  and the Lipschitz constants of $u$, $v$,   such that
\[ u= v \Txt{in $[1/2 - \de, 1/2 + \de] \times [T_0, T_0+\de]$.}\]
\end{Lemma}
\begin{proof} We denote by $L$ a Lipschitz constant of both $u$, $v$ in $R$, we further denote by $M$ a Lipschitz constant of $H(s, \cdot)$ in $[-L,L]$, for any $s \in [0,1]$, see assumption {\bf (H4)}, that we can assume greater than $3$, so that
\[|H(s,u'(s,t)) - H(s,v'(s,t))| \leq M \, |u'(s,t)) -v'(s,t))|\]
 for  a.e. $ (s,t) \in R$. We then have
  \begin{eqnarray*}
    0 &=& u_t(s,t) + H(s,u'(s,t)) - v_t(s,t) - H(s,v'(s,t)) \\
    &\geq& (u_t(s,t) - v_t(s,t)) - M \, |u'(s,t)-v'(s,t)|
  \end{eqnarray*}
a.e. $(s,t) \in R$.  Consequently $(u-v)$ and similarly $(v-u)$ are a.e. and viscosity subsolutions to the equation
\begin{equation}\label{propaga11}
w_t - M \, |w'| =0 \Txt{in $R$}
\end{equation}
attaining the value $0$ on $[0,1] \times \{T_0\}$ and $u-v$ (resp. $v-u$) on the rest of the parabolic boundary of $R$. We know that the solution  of \eqref{propaga11} with these boundary conditions   is given by
\[(s,t) \mapsto  \max \{(u-v)(s^*,t^*) \mid (s^*,t^*) \in  \partial_p \big ((0,1) \times (T_0,t) \big ), \, |(s,t)-(s^*,t^*)| \leq M \, (t - t^*)\}. \]
We take
\[(s,t) \in (1/2-1/M,1/2 + 1/M) \times (T_0,T_0+\de)=:U,\]
where $\de$ is a positive constant to be determined, we recall that  $M$ has been taken greater that $3$.  If $(s^*,t^*)$ is in  the lateral part of the parabolic boundary of $R$ then
\[|(s,t)-(s^*,t^*)| \geq \min \{ 1-s,s \} \geq \frac 12 - \frac 1M= M \, \left ( \frac 1{2M} - \frac1{M^2} \right ).\]
It is then enough to take
\[\de < \frac 1{2M} - \frac1{M^2} < \frac 1M\]
to see that the lateral boundary does not have influence in the above formula of the solution at $(s,t) \in U$. This shows that the solution is vanishing in $U$, since $u-v$ and $v-u$ are both less than or equal the solution by the comparison principle, we deduce
\[u=v \Txt{in $U$.} \]
Since $U \supset [1/2 - \de, 1/2 + \de] \times [T_0, T_0+\de]$, we get the assertion.
\end{proof}

\smallskip

We derive:

\begin{Corollary}\label{lemfaberdue} Let $u_0$, $v_0$ be  admissible Lipschitz continuous boundary data for \eqref{HJloc} in $\partial^-_p R$ and $\partial^+_p R$,
respectively, with $u_0=v_0$ on
 $[0,1] \times \{T_0\}$.  Then the merge of the two functions is
 admissible  on $\partial_p \big ([0,1] \times [T_0, T_0+\de] \big )$, for
 a suitable  constant $\de > 0$ solely depending on  the Lipschitz constants of $u_0$, $v_0$ and $H$.
\end{Corollary}
\begin{proof}    We
denote  by $u$, $v$ the maximal (sub)solutions agreeing with $u_0$ on
$\partial^-_p R$ and $v_0$ on $\partial^+_p R$, respectively. We derive from Proposition \ref{maxistar} and Remark  \ref{fabro} that their Lipschitz constants depend on those of $u_0$, $v_0$ and clearly on $H$.   By
Lemma \ref{propaga1} $u$ and $v$ coincide in
\[ [1/2-\de,1/2 + \de] \times [T_0, T_0+\de], \]
for a suitable $\de$,  we define a new solution $w$ setting
\[ w =  \left \{  \begin{array}{cc}
         u & \quad\hbox{on $[T_0, 1/2- \de) \times [T_0, T_0+\de]$}   \\
          v & \quad\hbox{on $( 1/2+\de,1] \times [T_0,T_0+ \de]$} \\
          u=v & \quad\hbox{on $[1/2 - \de, 1/2+\de) \times [T_0,
          T_0+ \de]$}\\
          \end{array} \right . \]
  The function $w$ coincides with $u_0$ on $\partial^-_p  \big ([0,1] \times [0,\de] \big )$ and with $v_0$ in $\partial^+_p \big ([0,1] \times [0,\de] \big )$.   This proves the assertion
\end{proof}

\bigskip

\section{A semidiscrete equation} \label{semidiscrete}
We set
\begin{eqnarray*}
  \F  &=&    \UC((\Gamma \times \{T_0\}) \cup (\VV \times [T_0,T_1)) \\
  R &=& (0,1) \times (T_0,T_1)  \Txt{with $T_1 \leq + \infty$.}
\end{eqnarray*}

 We define, for any $x \in \VV$, the operator  $F_x: \F \to \UC([T_0,T_1))$ in the interval $[T_0,T_1)$ through two steps:
  \begin{itemize}
     \item[--] Given  $u \in \F$ and  $\ga \in \Gamma_x$
we indicate by $(s,t) \mapsto F_\ga[u](s,t)$,
 the maximal among the  uniformly continuous  subsolutions to \eqref{HJg} in $R$ with trace less than or equal to  the merge of $u \circ \ga$ and $u(\ga(0),t)$ on $\partial_p^- R$.  Note that  $F_\ga[u]$ is uniformly continuous by Proposition \ref{maxistaruni}, and in addition Lipschitz continuous if  $u$ is Lipschitz continuous by Proposition \ref{maxistar};
     \item[--]  we  set
\[F_x[u] = \min_{\ga \in \Gamma_x} F_\ga[u](1,\cdot) .\] \end{itemize}
\smallskip

Note that
\begin{equation}\label{mora}
  F_x[u](T_0)= u(x,T_0) \Txt{for any $x\in \VV$, $u \in \F$}
\end{equation}

\smallskip

We directly derive from the definition of $F_x$:

\begin{Lemma}\label{lemfaber0} For any $u \in \F$, $x \in \VV$
\begin{itemize}
    \item[{\bf (i)}] $F_x[u + a]= F_x[u] + a$ for any  $a \in \R$;
    \item[{\bf (ii)}] if $v \in \F$ with  $v \geq u$ then  $F_x [v] \geq F_x[u]$ in $[T_0,T_1)$.

\end{itemize}
\end{Lemma}

\medskip

\subsection{ Definition of the problem and comparison result}    Given a flux limiter $x \mapsto c_x$ on $\VV$, we consider the semidiscrete equation
\begin{equation}\label{discr} \tag{{\bf Discr}}
  u(x,t) =  G[F_x[u],c_x](t) .
\end{equation}
See Section \ref{AAA} for the definition of the operator $G$.  By uniformly continuous  {\em (sub / super) solution} of it in the interval $(T_0,T_1)$, we mean a  function
\[v \in \UC((\Gamma \times \{0\}) \cup (\VV \times [T_0,T_1)))\]   which satisfies  pointwise the (in)equalities in \eqref{discr} for any $(x,t) \in \VV \times (T_0,T_1)$.  If
\[v(x,t) <  G[F_x[v],c_x](t)   \Txt{for any $x \in \VV$, $t \in (T_0,T_1)$}\]
then we say that $v$ is a {\em strict} subsolution.

\smallskip

\begin{Lemma}\label{lemuni} Let $u$ be an  uniformly continuous subsolution  to \eqref{discr} in $(T_0,T_1)$ , then $u - \eps \, (t-T_0)$ is a strict subsolution for any $\eps >0$.
\end{Lemma}
\begin{proof} By the very definition of $F_\ga$,  the function
$F_\ga[u](s,r)  - \eps \, (r-T_0) $ is a strict subsolution of \eqref{HJg} in $R$
for any arc $\ga$, then we apply Lemma \ref{lemfabertre}  to $F_\ga[u](s,r)  - \eps \, r $  to get
\[ F_\ga[u](s,r)  - \eps \, (r-T_0) < F_\ga[u  - \eps \, (t-T_0)](s,r) \Txt{in $\ov R \setminus \partial^-_p R$.}\] and in particular
\[F_\ga[u](1,r)  - \eps \, (r-T_0) < F_\ga[u  - \eps \, (t-T_0)](1,r) \Txt{for any $\ga$, $r \in (T_0,T_1)$.}\]
 This implies
 \begin{equation}\label{lemuni1}
 F_x[u](r)  - \eps \, (r-T_0) < F_x[u  - \eps \, (t-T_0)](r) \Txt{for any $r \in (T_0,T_1)$.}
\end{equation}
 Since $u$ is subsolution to
\eqref{discr}, we derive from \eqref{lemuni1} and  the definition of $G$, see \eqref{defG},  that for any $(x,t_0) \in \VV \times (T_0,T_1)$, $r \in (T_0, t_0]$
\begin{eqnarray*}
  u(x,t_0) - \eps \, (t_0-T_0)  &\leq&  G[F_x[u],c_x](t_0) - \eps \, (t_0-T_0)  \\ &\leq& F_x[u](r) + c_x \,(t_0-r) - \eps \, (t_0-T_0)  \\
 &\leq& F_x[u](r) + c_x \,(t_0-r) - \eps \, r  \\ &<& F_x[u  - \eps \,(t-T_0)] (r) + c_x \, (t_0-r).
\end{eqnarray*}
This shows the assertion provided that
\[G[F_x[u  - \eps \,(t-T_0)],c_x](t_0) = F_x[u  - \eps \,(t-T_0)] (r) + c_x \, (t_0-r) \]
for some $r \in (T_0,t_0]$, otherwise we have
\[ G[F_x[u  - \eps \,(t-T_0)], c_x](t_0) = u(x,T_0)+ c_x \, (t_0 - T_0).\]
In this case we exploit that $u$ is subsolution of \eqref{discr} in $(T_0,T_1)$ and \eqref{mora} to obtain
\begin{eqnarray*}
  u(x,t_0) - \eps \, (t_0 -T_0) &<& G[F_x(u),c_x](t_0) \leq u(x,T_0)+ c_x \, (t_0-T_0) \\
   &=& G[F_x[u  - \eps \,(t-T_0)],c_x](t_0).
\end{eqnarray*}
This concludes the proof.

\end{proof}

\smallskip

\begin{Theorem}\label{unicum} Let $u$, $v$ be uniformly continuous  sub and supersolution to \eqref{discr},  respectively, in $(T_0,T_1)$ with
\[ u(\cdot,T_0) \leq v(\cdot,T_0) \qquad\hbox{in $\VV$} \]
   then
\[ u \leq v  \qquad\hbox{on $\VV \times [T_0,T_1)$.}\]
\end{Theorem}
\begin{proof} We assume for
purposes   of contradiction   that $u > v$ for some
$(y,t_0) \in \VV \times (T_0,T_1)$.   By taking $\eps$ small, we
get
\[u(y, t_0)- \eps \, (t_0 -T_0)  -v(y,t_0)>0.\]
 By replacing  $u $ by  $u- \eps \, (t-T_0)$,  for such a small $\eps$, and bearing in mind Lemma
\ref{lemuni},  we can therefore assume without loosing
generality that $u$ is a strict subsolution.

We denote by $(y,t^*)$, $t^* \in (T_0,t_0]$, a maximizer of $u - v$ in
$\VV \times [T_0,t_0]$, it   does exist because $\VV
\times [T_0,t_0]$ is a compact set and $u(y,\cdot)-v(y,\cdot)$ is  continuous.
We have
\begin{equation}\label{unicum1}
 a:= u(y,t^*) - v(y,^*) > 0.
\end{equation}
Let $r$ be a time with $r \in [T_0, t^*]$ such that
\[ v(y,t^*) \geq G[ F_y[v],c_y] (t^*)=  F_y[v](r)  +c_y \, (t^*-r) , \]
 then taking into account that
 \begin{eqnarray*}
   v(x,t) +a &\geq & u(x,t) \Txt{in $\VV \times [T_0,t_0]$}  \\
   u(y,t^*) &<& F_y[u](r) +c_y \, (t^*-r)
 \end{eqnarray*}
we can use Lemma \ref{lemfaber0} {\bf (i)}, {\bf (ii)} to  get
\begin{eqnarray*}
  v(y,t^*) + a &\geq& F_y[v](r) +c_y \, (t^*-r) + a =  F_y[ v + a](r) +c_y \, (t^*-r) \\
   &\geq& F_y[u](r) + c_y \, (t^*-r) > u(y,t^*),
\end{eqnarray*}
 in contrast with \eqref{unicum1}.
\end{proof}

\medskip

\subsection{Links between semidiscrete and Hamilton--Jacobi equation}

We proceed linking  \eqref{discr} to (HJ$\Gamma$). We consider a time interval $[T_0,T_1) \subset \R^+$ with $T_1 \leq + \infty$.

\begin{Proposition}\label{fundasub}  Let $u$ be an uniformly   continuous subsolution to (HJ$\Gamma$) in $[T_0,T_1)$  then the trace of $u$ on $(\Gamma \times \{T_0\}) \cup (\VV \times [T_0,T_1))$ is a subsolution of \eqref{discr} in  $[T_0,T_1)$.
\end{Proposition}
\begin{proof} We  apply the definition of subsolution to (HJ$\Gamma$) and of $F_x[u]$ to get for every $x \in \VV$, $t \in [T_0,T_1)$, every $C^1$ supertangent $\varphi$ to $u(x,\cdot)$ at $t$
\begin{eqnarray*}
  u(x,t) &\leq& F_x[u](t)  \\
 \frac d{dt} \varphi (t) &\leq& c_x.
\end{eqnarray*}
We deduce from Lemma \ref{proprovvi}
\[ u(x,t) \leq G[F_x[u],c_x](t) \Txt{for any $(x,t) \in \VV \times [T_0,T_1)$.}\]
\end{proof}

\begin{Proposition}\label{fundasuper} Let $u$ be an uniformly  continuous supersolution to (HJ$\Gamma$) in $[T_0,T_1)$ then the trace of $u$ on $(\Gamma \times \{T_0\}) \cup (\VV \times [T_0,T_1))$ is a supersolution of \eqref{discr} in $[T_0,T_1)$.
\end{Proposition}
\begin{proof}
We fix a vertex $x$ and an arc $\ga \in \Gamma_x$.
 According to Proposition \ref{ginetto},  to prove that
 \[u(x,t) \geq G[F_x[u],c_x](t) \Txt{for any $t \geq 0$}\]
 it is enough  showing  that  for any $t_0 \in [T_0,T_1)$ where there is a   $C^1$ subtangent $\varphi$ to $u(x,\cdot)$ with  $\frac d{dt} \varphi(t_0) < c_x$, one has
\begin{equation}\label{funda100}
u(x,t_0) \geq   F_x[u](t_0).
   \end{equation}

\smallskip

\noindent{\bf Step 1.} \; We fix $a$ with
\begin{equation}\label{funda0}
 0 > c_x > a > \frac d{dt} \varphi(t_0)
\end{equation}
and $p_0 \in \R$  such that
\begin{equation}\label{funda1}
a + H_\ga(1,p_0) < 0,
\end{equation}
which implies
\begin{equation}\label{funda4}
\frac d{dt} \varphi(t) + H_\ga(s,p_0) < 0  \Txt{for $(s,t) \in \ov Q$ close to $(1,t_0)$}
\end{equation}
We can assume without loosing generality that $\varphi$ is a strict subtangent such that  $u(x,t_0)= \varphi (t_0)$ and
\begin{equation}\label{funda21}
 u(x, t) >  \varphi(t) \Txt{for $t \neq t_0$, $t$ close to $t_0$.}
\end{equation}
We can therefore determine $\de >0$ such that $t_0 + \de < T_1$ and  both \eqref{funda4}, \eqref{funda21} hold true for $(s,t) \in [1- \de,1] \times  [t_0- \de, t_0 + \de]$; in addition, since
\[s \mapsto u(\ga(s), t_0 - \de) -   \varphi(t_0 -\de) - p_0 \, (s-1) \]
is continuous and positive at $s=1$  by \eqref{funda21}, we can also  get
\begin{equation}\label{funda211}
 u(\ga(s), t_0 - \de) >   \varphi(t_0 - \de) + p_0 \, (s-1) \Txt{for $s \in [1-\de,1]$.}
\end{equation}
If we assume  by contradiction  that
\begin{equation}\label{funda5}
   F_x[u](t_0) > u(x,t_0)
\end{equation}
then  we further have, up to shrinking $\de$
\begin{equation}\label{funda5bis}
   F_\ga[u](s,t) > u(\ga(s), t) \Txt{for  $(s,t) \in [1- \de,1]
\times [t_0- \de, t_0 + \de] $.}
\end{equation}

\smallskip
\noindent {\bf Step 2} \;  We claim that the function
\begin{equation}\label{funda6}
    \phi(s,t) = \varphi (t)+ p_0 \, (s- 1) .
\end{equation}
is a  subtangent, constrained to $\ov Q$,  to $u \circ \ga$ at $(1,t_0)$.
By applying Proposition \ref{xanax} to $s_0= 1 -\de$,  $c=c_x$, the supersolution $u \circ \ga $, the sets
\[A= (0,1-\de)\times (t_0-\de,t_0+\de) \Txt{and} \qquad B = (1-\de,1)\times (t_0 - \de,t_0+\de),\]
the maximal subsolutions, denoted by $v$, $w$ to \eqref{HJg}, with trace  less than or equal to $u \circ \ga $ on $\partial^-_p A$, $\partial^+_p B$, respectively,    we derive
\begin{equation}\label{funda1000}
  G[u(\ga(1-\de),\cdot), c_x](t) \geq \min \{G[v(1-\de,\cdot),c_x],G[w(1-\de,\cdot),c_x]\}(t)
\end{equation}
for any $t \in [t_0-\de,t_0+\de]$. We further deduce from \eqref{funda5bis} and the relation $v \geq  F_\ga[u]$ in $A$
 \[ v(s,t) > u(\ga(s),t) \Txt{for $(s,t) \in \{1-\de\} \times [t_0 -\de,t_0 + \de]$}\]
 and   consequently by Lemma \ref{provvi}
 \[ G[v(1-\de,\cdot),c_x] (t)> G[u(\ga(1-\de),\cdot), c_x](t) \txt{for $t \in [t_0-\de,t_0 + \de]$} .\]
We then have by \eqref{funda1000}
\begin{equation}\label{funda5000}
 G[u(\ga(1-\de),\cdot), c_x](t)\geq G[w(1-\de,\cdot),c_x](t) \Txt{ for $t \in [t_0- \de,t_0+ \de]$.}
\end{equation}
We know  by \eqref{funda21}, \eqref{funda211} that $\phi \leq u \circ \ga$ on $\partial_p^+ B$,  we further know by \eqref{funda4} that  it is a subsolution of \eqref{HJg} in $B$.  We derive taking into account the maximality property of $w$
\[\phi(1-\de,t) =G[\phi(1-\de,\cdot),c_x] \leq w(1-\de,t) \Txt{for $t \in [t_0-\de,t_0 + \de]$}\]
which implies
\begin{eqnarray*}
  \phi(1-\de,t) &=& G[\phi(1-\de,\cdot),c_x] \leq G[w(1-\de,\cdot),c_x](t) \\
  &\leq& G[u(\ga(1-\de),\cdot), c_x](t) \leq u(\ga(1-\de),t)
\end{eqnarray*}
for $t \in [t_0-\de,t_0 + \de]$. Summing up, $u\circ \ga \geq \phi$ on the whole of the parabolic boundary of $B$. We derive from the comparison principle  given in Theorem \ref{nobarles} that
\[ u(\ga(s),t) \geq \phi(s,t) \Txt{for $(s,t) \in \ov B$}\]
Taking into account that  $u= \phi$ at $(1,t_0)$, we conclude that  the function $\phi$ is a  subtangent, constrained  to $\ov Q$,  to $u \circ \ga$ at $(1,t_0)$, as was claimed.

\smallskip

\noindent {\bf Step 3} \;  We therefore reach a contradiction with $u$ being supersolution and $\frac d{dt} \varphi(t_0) < c_x$, since $\ga$ is an arbitrary  arc in $\Gamma_x$ and  by \eqref{funda1}
\[ \phi_t(1,t_0) + H_\ga(1,\phi'(1,t_0))  < 0.\]
This shows  \eqref{funda100} and ends   the proof.
\end{proof}

\smallskip

We recall that
\[  R= (0,1) \times (T_0,T_1).\]

\begin{Proposition}\label{fundasol} Let $u:\Gamma \times [T_0,T_1) \to \R$ be an uniformly continuous function such that  $u \circ \ga$ is solution to \eqref{HJg} in $(0,1) \times (T_0,T_1)$ for any arc $\ga$, and the trace of $u$ on $(\Gamma \times \{0\}) \cup (\VV \times [T_0,T_1))$ solves \eqref{discr} in $(T_0,T_1)$, then $u$ is solution of (HJ$\Gamma$) in $(T_0,T_1)$.
\end{Proposition}

 \begin{proof}   It is clear from the assumptions that $(x,t) \mapsto u(x,t)$ is subsolution to (HJ$\Gamma$) in $[T_0,T_1)$.  We fix  $x_0 \in \VV$.  It is left checking condition {\bf (ii)} in the definition of supersolution. Let  $t_0 \in (T_0,T_1)$, we select $\ga \in \Gamma_{x_0}$  with
\[ F_\ga[u](1,t_0)=F_{x_0}[u](t_0).\]
Let $\psi$  be a  $C^1$  subtangent to $u(x_0,\cdot)$ at $t_0$  with $\partial_t\psi(t_0) <  c_{x_0}$, therefore by Proposition \ref{overpfb2}
 \[u(x_0,t_0)= F_{x_0}[u](t_0)=F_\ga[u](1,t_0).\]
 In addition
 \[u(\ga(s),t) \leq F_\ga[u](s,t) \Txt{ for $(s,t)  \in \ov R$}\]
 so that   any   subtangent $\phi$, constrained to $\ov R$ , to $u$ at $(1,t_0)$ is also subtangent  to $F_\ga[u]$.
 If we now  assume by contradiction that
\[\phi_t(1,t_0) + H_\ga(1, \phi'(1,t_0)) <  0\]
we can construct via Perron--Ishii method a subsolution to \eqref{HJg}  strictly greater than $F_\ga[u](s,t)$  in a neighborhood of $(1,t_0)$ intersected with $\ov R$.   This contradicts the maximality of $F_\ga[u]$,  and concludes  the proof.
 \end{proof}

\smallskip

The above argument also allows  pointing out a property of solutions we will use to prove  some stability results.

\smallskip
\begin{Corollary}\label{happy} Assume  $u$ to  be an uniformly   continuous solution of (HJ$\Gamma$) in $(T_0,T_1)$, and consider  $x_0 \in \VV$. Assume that   there is a  $C^1$   subtangent $\varphi$ to $u(x_0,\cdot)$ at $t_0$, for some $t_0 \in (T_0,T_1)$  with $ \frac d{dt} \varphi (t_0) < c_{x_0}$. Then  there exists $\ga \in \Gamma_{x_0}$ such that $u(x_0,t_0)$ coincides with the maximal subsolution of \eqref{HJg} in $R$, with trace $u\circ \ga$ on $[0,1] \times \{T_0\}$  and $u(\ga(0),t)$ on $\{0\} \times [T_0,T_1)$, computed at $(1,t_0)$.
\end{Corollary}
\begin{proof} We derive from the assumptions, arguing as in  Proposition \ref{fundasol}, that
\[ u(x_0,t_0)= F_{x_0}[u](t_0).\]
To conclude the proof it is then enough to take $\ga \in \Gamma_{x_0}$ with
\[F_{x_0}[u](t_0)= F_\ga[u](1,t_0)\]
and recall the definition of $F_\ga[u]$.

\end{proof}

\smallskip

We further show that Proposition \ref{melaton} can be extended to (sub)solutions of (HJ$\Gamma$).

\smallskip

\begin{Corollary}\label{superhappy}  Let $t^* \in (T_0,T_1)$. Let $u$ be an uniformly  continuous function from $\Gamma \times [T_0,T_1)$ to $\R$. Assume $u$ to be (sub)solution of (HJ$\Gamma$) in $ (T_0,t^*)$ and in $(t^*,T_1)$. Then $u$ is (sub)solution in $(T_0,T_1)$ as well.
\end{Corollary}
\begin{proof} We  set $S= (0,1) \times (T_0,t^*)$. Taking into account Proposition \ref{melaton}, the assertion is immediate in the case of subsolutions. In the case of solutions,  we need just  checking property {\bf (ii)} in the definition of supersolution at $t^*$. Given $x \in \VV$, we therefore assume that there is  a  $C^1$  subtangent $\psi$  to $u(x,\cdot)$  at  $t^*$ with
\begin{equation}\label{super1}
  \frac d{dt} \psi(t^*)< c_x.
\end{equation}
 We claim  that there exists a sequence $t_n < t^*$, $t_n \to t$ and $C^1$ subtangents $\psi_n$ to $u(x,\cdot)$ at $t_n$ with
 \begin{equation}\label{happy2}
 \frac d{dt} \psi_n(t_n)< c_x.
 \end{equation}
If on the contrary, one can find an interval $( t^*-\de,t^*)$, for some $\de >0$, such that
\[  \frac d{dt} \varphi(t) \geq c_x\]
for any $t \in (t^*-\de,t^*)$, any $C^1$ subtangent $\psi$ to $u(x,\cdot)$ at $t$, then, according to Lemma \ref{decrescebis}
\[u(\cdot,t) - c_x \, (t -t^*)\]
is nondecreasing in $[t^*-\de, t^*]$. In contrast with \eqref{super1}. By applying  Corollary \ref{happy} to the $t_n$'s satisfying \eqref{happy2}, we find  a sequence $\ga_n$ of arcs in $\Gamma_x$ such that   $u(x,t_n)$ coincides with the maximal subsolution of \eqref{HJg}, with $\ga$ replaced by $\ga_n$,  in $S$   with trace equal to $u\circ \ga_n$ on $[0,1] \times \{T_0\}$  and $u(\ga_n(0),t)$ on $\{0\} \times [T_0,t^*)$, computed at $t_n$. The sequence $\ga_n$ be must definitively constant, equal, say, to $\ga \in \Gamma_x$. Accordingly, we have that
\[u(x,t^*)= \lim_n w(1,t_n),\]
where we have denoted by $w(s,t)$ the maximal subsolution to \eqref{HJg} in $S$ with trace $u \circ \ga$ on $\partial_p^- S$. Now assume for purposes of contradiction that there is a $C^1$  subtangent $\phi$, constrained to $[0,1] \times [T_0,T_1)$,  to $u \circ \ga$ at $(1,t^*)$,  with
\begin{equation}\label{super3}
  \phi_t(1,t^*) + H_\ga(1,\phi'(1,t^*)) <0.
\end{equation}
We can further assume that
\[ \phi(1,t^*)= u(x,t^*) \txt{and} \quad \phi < u\circ \ga \txt{in $U \cap \big ([0,1] \times [T_0,T_1) \big ) \setminus \{(1,t^*)\}$,}\]
where $U$ is a suitable neighborhood of $(1,t^*)$ in $\R^2$. Exploiting the maximality of $w$, we deduce that
\[ w > \phi \Txt{ in $\big (U \cap S\big ) \setminus \{(1,t^*)\}$,}\]
so that using \eqref{super3} we can construct via Perron--Ishii method a subsolution to \eqref{HJg} agreeing with $w$ in $\partial_p^- S$ and strictly greater to $w$ at points of $S$ suitably close to $(1,t^*)$.  We have therefore reached a contradiction, which shows that \eqref{super3} is not possible. This ends the proof.
\end{proof}

\bigskip

\section{Well posedness of {\bf HJ$\Gamma$}}\label{well}

As already clarified in the Introduction and Section \ref{defsol}, the well posedness is relative to the time dependent Hamilton--Jacobi equation coupled with a continuous initial datum plus a flux limiter  at the vertices.\\

\subsection{Comparison result}

We start  with a comparison result based on Theorem \ref{unicum} and the links established between Hamilton--Jacobi equation and semidiscrete problem.

\smallskip

\begin{Theorem}\label{unicumque} Let $u$, $v$ be uniformly continuous sub and   supersolution  of (HJ$\Gamma$) respectively, in  $(T_0,T_1)$ with $u(\cdot,T_0) \leq v(\cdot,T_0)$ in $[0,1]$,  then $u \leq v$ in $\Gamma \times [T_0,T_1)$.
\end{Theorem}
\begin{proof}  The trace of $u$ and $v$ on $ (\Gamma \times \{T_0\}) \cup (\VV \times [T_0,T_1))$ are subsolution and supersolution to \eqref{discr} by Propositions \ref{fundasub}, \ref{fundasuper}, respectively.  We then invoke Theorem \ref{unicum} to get
\[ u \leq v \Txt{in $(\Gamma \times \{T_0\}) \cup (\VV \times \R^+)$}.\]
We apply the comparison result in Theorem \ref{nobarles}  to finally obtain
\[ u \circ \ga \leq v \circ \ga   \Txt{in $[0,1] \times [T_0,T_1)$, for any $\ga$.}\]
This proves the assertion.
\end{proof}

We derive :

\begin{Corollary}\label{uniqlo} Let $g$ be a continuous  datum on $\Gamma$, there exists at most one uniformly continuous solution to  (HJ$\Gamma$) in $(T_0,T_1)$ agreeing with $g$ at $t=T_0$.
\end{Corollary}

We further derive a first stability result.

\begin{Corollary}\label{corunicumque} Let $g_n$ a sequence of continuous functions in $\Gamma$ uniformly converging to a function $g$. Assume that there exist (unique) uniformly continuous solutions   $u_n$  of  (HJ$\Gamma$) in $(T_0,T_1)$ with initial data $g_n$  at $t=T_0$ and flux limiter $c_x$. Then  the $u_n$'s are uniformly convergent in $\Gamma \times [T_0,T_1)$ to a solution  $u$  of  (HJ$\Gamma$) in $(T_0,T_1)$ agreeing with $g$ at $t=T_0$, and
\[|u_n-u|_\infty \leq  |g_n-g|_\infty  \Txt{for every $n$.} \]
\end{Corollary}
\begin{proof}
We have
\[g_m - |g_n - g_m|_\infty \leq g_n \leq g_m + |g_n - g_m|_\infty \Txt{in $\Gamma$, for $n$, $m$ in $\N$.}\]
We derive from the comparison principle in Theorem \ref{unicumque} that
\[ u_m - |g_n - g_m|_\infty \leq u_n \leq u_m + |g_n - g_m|_\infty \Txt{in $\Gamma \times [T_0,T_1)$,}\]
which implies
\[ |u_n-u_m|_\infty \leq |g_n-g_m|_\infty\]
then $u_n$ uniformly converges in $\Gamma \times [T_0,T_1)$ to a function $u$ and
\[|u_n-u|_\infty =\lim_m |u_n-u_m|_\infty \leq \lim_m |g_n-g_m|_\infty= |g_n-g|_\infty.\]
Taking into account the usual stability results in viscosity solutions theory, to prove that $u$ is solution of (HJ$\Gamma$) in $(T_0,T_1)$, it is enough to check that it satisfies  the condition {\bf (ii)} in the definition of supersolution.

We assume that for a given $(x_0,t_0) \in \VV  \times (T_0,T_1)$   there exists a $C^1$ subtangent $\varphi$ to $u (x_0,\cdot)$ at $t_0$ with $ \frac d{dt}\varphi(t_0) < c_{x_0}$. We derive that $\varphi$ is subtangent to $u_n (x_0,\cdot)$ at points $t_n$ with $t_n \to t_0$, and
 \[ \frac d{dt} \varphi(t_n) < c_{x_0} \Txt{for $n$ large enough.}\]
  We further  derive from   Corollary \ref{happy}, applied to $u_n$ and $(x_0,t_n)$, that
 there is $\ga_n \in \Gamma_{x_0}$ such that $u_n(x_0,t_n)$ coincides with the maximal subsolution of \eqref{HJg}, with  $\ga_n$ in place of $\ga$, with trace $g_n\circ \ga_n$ on $[0,1] \times \{0\}$  and $u_n(\ga_n(0),t)$ on $\{0\} \times \R^+$, computed at $(1,t_0)$.

 Since $\Gamma_{x_0}$ contains finite elements, we can assume, up to further extracting a subsequence, that the sequence $\ga_n$ is constant, equal, say, to $\ga$.
 We then invoke Corollary \ref{2021bis} to conclude that $u(x_0,t_0)$ coincides with the maximal solution of \eqref{HJg} with trace $g\circ \ga$ on $[0,1] \times \{0\}$  and $u(\ga(0),t)$ on $\{0\} \times \R^+$, computed at $(1,t_0)$. Arguing as in Proposition \ref{fundasol}, we deduce that
\[\psi_t(1,t_0) + H_\ga(1, \psi'(1,t_0)) \geq  0\]
for any  subtangent $\psi$, constrained to  $[0,1] \times  \ov{(T_0,T_1)}$, to $u \circ \ga$ at $(1,t_0)$. This concludes the proof.

\end{proof}

\medskip

\subsection{Existence of solutions and stability properties}

In this section  we show the existence of a solution to  (HJ$\Gamma$) in $(T_0,T_1)$, with $T_1 \leq + \infty$, and consequently to \eqref{discr}, coupled with any continuous initial datum $g$ and any flux limiter $c_x$.   We first assume the initial datum $g$ to be Lipschitz continuous in $\Gamma$, and set
\begin{equation}\label{trump9}
  M_0 = \left (\max_\ga \, \min\{m \mid H_\ga(s,(g \circ \ga)') \leq m\}\right ) \, \vee \, \left ( \max_{x \in \VV } |c_x| \right )
\end{equation}

\medskip

\begin{Theorem}\label{faber}  For any given   Lipschitz continuous initial datum $g$ at $t=T_0$, any flux limiter $c_x$,  there exists a Lipschitz continuous solution $u$ of (HJ$\Gamma$) in $(T_0,T_1)$ with
\begin{equation}\label{trump11}
  \Lip \big (  u \circ \ga(s,\cdot)\big )  \leq M_0   \Txt{for any $\ga$, any $s \in [0.1]$.}
\end{equation}
\end{Theorem}
\begin{proof} We define
\[ \mathcal T= \{T > 0 \mid \exists \; \hbox{Lip. sol. $u$ to (HJ$\Gamma$) with $g$, $c_x$, in $(T_0,T)$ sat. \eqref{trump11}}\}.\]
We first  prove that $\mathcal T \neq \emptyset$. We set $R=(0,1) \times (T_0,T_1)$.

\smallskip

\noindent {\bf Step 1.} \;
 Given any arc $\ga$,  let
 $ v_\ga$ be  the maximal subsolution of
\eqref{HJg} in $R$ with initial datum $g \circ \ga$ at $t =T_0$, we
set
\begin{eqnarray*}
 v(x,t) &=& \min_{\ga \in \Gamma_x}  v_\ga(1,t) \Txt{for any
$x \in \VV$, $t \in (T_0,T_1)$} \\
  v(y,T_0) &=& g(y)  \Txt{for any
$y \in \Gamma$.}
\end{eqnarray*}

\smallskip

\noindent {\bf Step 2.} \;
We have
\[F_{\ga}[v](1,t) \leq v_\ga(1,t) \Txt{in $[T_0,T_1)$, for any $\ga$.}\]
Since, given $x \in \VV$, $t \in [T_0,T_1)$
\[v(x,t) = v_{\ga^x}(1,t) \Txt{for a suitable $\ga^x \in\Gamma_x$}\]
  we deduce  that
  \begin{equation}\label{pasqua3}
v(x,t) \geq F_{\ga^x}[v](1,t) \geq F_x[v](t) \Txt{ for any $x \in \VV$, $t \in [T_0,T_1)$}
\end{equation}

\smallskip

\noindent {\bf Step 3.} \;
We define
\begin{eqnarray} \label{pasqua01}
 \ov u(x,t) &=& G[v(x,\cdot),c_x](t) \leq v(x,t) \Txt{$x \in \VV$, $t \in [T_0,T_1)$} \\
  \ov u(y,T_0) &=& g(y) \Txt{$y \in \Gamma$.} \label{pasqua000}
\end{eqnarray}
To sum up: $\ov u(x,t)$  from $\VV \times [T_0,T_1) \to \R$ has been defined through three steps:
\begin{itemize}
  \item[--] first, we have set $v_\ga(1,t)$, for $\ga \in \Gamma_x$,  as the maximal subsolution of  \eqref{HJg} in $R$ among those agreeing with $g \circ \ga$ at $t=T_0$;
  \item[--]  second, we have defined $v(x,t)= \min_{\ga \in \Gamma_x} v_\ga(1,t)$;
  \item[--] and, finally,  $\ov u(x,t)= G[v(x,\cdot),c_x](t)$.
\end{itemize}
We see through estimate   \eqref{trump1} that  $\Lip v_\ga(1,\cdot) \leq M_0$, see \eqref{trump9}, for any $\ga \in \Gamma_x$, and consequently $\Lip v(x,\cdot) \leq M_0$, so that we have, according to Remark \ref{cazzopapa}
\begin{equation}\label{pasqua1}
\Lip \ov u(x,\cdot) \leq M_0  \Txt{for any $x \in \VV$.}
\end{equation}

\smallskip

\noindent {\bf Step 4.} \;
We  apply, for any $\ga$,  Proposition \ref{lemfaberuno}  to the admissible boundary datum $v_\ga$  in $\partial^- R$  and derive, via \eqref{pasqua01} and the very definition of $v$,  that   $(g \circ \gamma,\ov u(\ga(0),\cdot))$ is an admissible datum for \eqref{HJg} on
$\partial_p^- R$, for any arc $\ga$,  and the same holds true for  $(g \circ\ga,\ov u(\ga(1),t))$  in $\partial_p^+ R$.

We denote by $u_1 \circ \ga$, $u_2 \circ \ga$ the corresponding maximal subsolutions of \eqref{HJg} with the above data taken on $\partial_p^- R$, $\partial_p^+ R$ , respectively. According to the estimates given in \eqref{trump3}, \eqref{trump4}, \eqref{trump5}, \eqref{trump6} and  \eqref{pasqua1}, we get for any $\ga$
\begin{equation}\label{pasqua55}
 \Lip (u_1 \circ \ga), \, \Lip (u_2 \circ \ga) \leq M_0  \vee  \big (\max \{|p| \mid H_\ga(s,p)  \leq M_0 \} \big ).
 \end{equation}
 and
 \begin{equation}\label{pasqua56}
  \Lip (u_1 \circ \ga(s,\cdot)), \, \Lip (u_2 \circ \ga(s,\cdot)) \leq M_0 \txt{for any $s \in [0,1]$}
\end{equation}
 By monotonicity of $G$, \eqref{pasqua3} and \eqref{pasqua01}, we further have
\begin{eqnarray}
  G[F_x[\ov u] ,c_x](t) &\leq& G[F_x[v],c_x](t) \label{pasqua4}  \\
 &\leq& G[v(x,\cdot),c_x](t)= \ov u(x,t) \nonumber
\end{eqnarray}
for any $x \in \VV$, $t \in [T_0,T_1)$.

\smallskip

\noindent {\bf Step 5.} \;
We   apply  Corollary \ref{lemfaberdue} to all arcs $\ga$ and deduce for a suitable $\de >0$, with $T_0 +\de < T_1$,   depending on the $H_\ga$ and the Lipschitz constants of  $ u_1\circ\ga$, $u_2 \circ \ga$, see \eqref{pasqua55}, that, for any $\ga$, the merge of all three functions $g
\circ \ga$, $\ov u(\ga(0),t)$, $\ov u(\ga(1),t)$ is admissible in
$\partial_p  \big ( (0,1) \times (T_0,T_0+\de) \big )$.  We denote by $u \circ \ga$ the corresponding solutions, with  $u \circ \ga: [0,1] \times [T_0,T_0 +\de] \to \R$. We have for any $x \in \VV$
\[F_x[ u](t) = F_x[\ov u](t)  \geq \ov u(x,t)= u(x,t) \txt{in $[T_0,T_0+\de]$} \]
 and accordingly
\begin{equation}\label{pasqua5}
G[F_x[ u],c_x](t) = G[F_x[\ov u],c_x](t) \geq \ov u(x,t)= u(x,t) \txt{in $[T_0,T_0+\de]$.}
\end{equation}
By  combining \eqref{pasqua4}, \eqref{pasqua5}, we obtain
\[ G[F_x[ u],c_x](t) = G[F_x[\ov u],c_x](t) =\ov u(x,t)= u(x,t).\]
This implies by Proposition \ref{fundasol} that $u$ is solution to (HJ$\Gamma$) in $\Gamma \times (T_0, T_0+\de)$. We therefore see, taking into account \eqref{pasqua56}, that $T_0 + \de \in\mathcal T$.

\smallskip

\noindent {\bf Step 6.} \; We set $T^* = \sup \mathcal T$, and proceed proving that there is a solution of (HJ$\Gamma$) in $[T_0,T^*)$ satisfying \eqref{trump11}.  We select an increasing sequence $T_n \in \mathcal T$ with $T_n \to T^*$, and denote by $u_n$ the corresponding solutions satisfying \eqref{trump11} in $ (T_0,T_n)$.
By the uniqueness result in Corollary \ref{uniqlo} we have
\[u_n(x,t)=u_{n+1}(x,t) \Txt{ for $(x,t) \in \Gamma \times (T_0,T_n)$, for any $n \in \N$,}\]
so that a solution $u$ in $[T_0,T^*)$ can be unambiguously defined via
\[u(x,t)=u_n(x,t) \Txt{for $n$ with $T_n >t$.}\]
Since all the $u_n$ satisfy \eqref{trump11}, the same holds true for $u$.

\smallskip

\noindent {\bf Step 7.} \; To conclude the proof, it is then enough proving that $T^*=T_1$. We assume for purposes of contradiction that $T^* < T_1$, and iterate the  construction preformed in the first part of the proof  in the interval $(T_n,T_1)$ starting from $u(\cdot,T_n)$ as initial datum at $t=T_n$, where $u$ and $T_n$ are defined as in  Step 6.

Since $u$ satisfies \eqref{trump11}, we get, according to Proposition \ref{hoje}
\begin{equation}\label{uniqlo1}
 \max_\ga \, \min\{m \mid H_\ga(s,(u \circ \ga)'(s,T_n) \leq m\} \leq M_0.
\end{equation}
Arguing as in the first part of the proof, we show that we can define a solution of (HJ$\Gamma$) with initial datum $u(\cdot,T_n)$ in an interval $(T_n,T_n +\de_n)$ with $T_n + \de_n <T_1$. By the gluing result given in Corollary \ref{superhappy}, we get altogether a solution in $(T_0,T_n + \de_n)$.

The crucial point is that, apart the restriction $T_n + \de_n <T_1$, $\de_n$ does not depend further on $n$. It is in fact the positive constant for which Corollary \ref{lemfaberdue}  holds true, and it depends on $H_\ga$ and the Lipschitz constants  of the solutions of \eqref{HJg}, for any $\ga$,  in $(T_n,T_1)$ to be merged. Due to \eqref{uniqlo1}, they satisfy estimates as in \eqref{pasqua55}, independent, to repeat, of $n$.

To get a contradiction, it is then enough to take, for a suitable $n$, $\de_n$ with
\[T_1 > T_n + \de_n > T^*.\]
This ends the proof.
\end{proof}

\smallskip

\begin{Remark} \label{postpasqua}   We denote by $u$ the unique Lipschitz continuous solutions of (HJ$\Gamma$).    According to \eqref{trump7}, \eqref{trump8}, we see that
\[ \Lip (u \circ \ga) \leq M_0 \vee \big (\max \{|p| \mid H_\ga(s,p)  \leq M_0 \} \big ) \txt{for any $\ga$.}\]
The Lipschitz constant of the solution therefore only depends on the the Hamiltonians $H_\ga$ and the Lipschitz constant of the initial datum.
\end{Remark}

\medskip

We proceed proving, through Corollary \ref{corunicumque},  existence of solutions of (HJ$\Gamma$) for any continuous initial datum.

\begin{Proposition}\label{nichil} For any continuous  initial datum $g$, any flux limiter $c_x$,  there exists an uniformly  continuous solution $u$ of (HJ$\Gamma$) in $(T_0,T_1)$. \end{Proposition}
\begin{proof} Let $g_n$ be a sequence of Lipschitz continuous function in $\Gamma$ uniformly converging to $g$,  see Lemma \ref{hopf} for the existence of such approximations. By applying Corollary \ref{corunicumque}, we see that the  unique Lipschitz continuous solutions $u_n$ to  (HJ$\Gamma$) in $(T_0,T_1)$ with initial datum $g_n$, which do exist by Theorem \ref{faber}, uniformly converge to an uniformly continuous solution of  (HJ$\Gamma$) in $(T_0,T_1)$ with initial datum $g$.

\end{proof}

\smallskip
We proceed giving a new version of Theorem \ref{unicumque}.

\begin{Theorem}\label{unicumquebis} Let $u$, $v$ be  continuous subsolution and   solution  of (HJ$\Gamma$) in $(T_0,T_1)$, respectively, with $u(\cdot,T_0) \leq v(\cdot,T_0)$, then $u \leq v$ in $\Gamma \times [T_0,T_1)$.
\end{Theorem}
\begin{proof} Assume by contradiction that there exists  $(x_0,t_0) \in \Gamma \times (T_0,T_1)$  with
\[u(x_0,t_0) > v(x_0,t_0).\]
According to Proposition \ref{nichil}, there exist uniformly continuous solutions $\ov u$, $\ov v$ to  (HJ$\Gamma$) in $ (t_0,T_0)$ with initial data $u(\cdot,t_0)$, $v(\cdot,t_0)$, respectively, at $t=t_0$. We set
\[ \widetilde u(x,t)= \left \{ \begin{array}{cc}
                        u(x,t) & \txt{ for $(x,t) \in \Gamma \times [T_0,t_0)$} \\
                        \ov u(x,t) & \txt{ for $(x,t) \in \Gamma \times (t_0,T_1)$}
                      \end{array} \right .\]
  and
                       \[ \widetilde v(x,t)= \left \{ \begin{array}{cc}
                        v(x,t) & \txt{ for $(x,t) \in \Gamma \times [T_0,t_0)$} \\
                        \ov v(x,t) & \txt{ for $(x,t) \in \Gamma \times (t_0,T_1)$}
                      \end{array} \right . \]
According to Corollary \ref{superhappy}, $\widetilde u$, $\widetilde v$ are uniformly continuous subsolution and solution to  (HJ$\Gamma$) in $(T_0,T_1)$, respectively,  but then the inequality $\widetilde u > \widetilde v$ at $(x_0,t_0)$ is in contrast with Theorem \ref{unicumque}.

\end{proof}

\smallskip

By summarizing the information gathered in  Theorem   \ref{faber}, Proposition \ref{nichil} and Theorem \ref{unicumquebis}, we can state:

\begin{Proposition}\label{orte}  For any continuous initial datum $g$ and flux limiter $c_x$, there exists one and only one continuous solution to (HJ$\Gamma$) in $(T_0,T_1)$. This solution is in addition uniformly continuous in $\Gamma \times [T_0,T_1)$. If $g$ is Lipschitz continuous, the solution is Lipschitz continuous as well.
\end{Proposition}

\smallskip

We now consider  a sequence of Hamiltonians $H^n_\ga$ for any arc $\ga$,  a sequence of continuous functions $g_n$ from $\Gamma$ to $\R$, and a sequence of  real numbers $c^n_x$, for any vertex $x$,  with
\[ c^n_x \leq \min_{\ga \in \Gamma_x} c^n_\ga,\]
where
\[  c^n_\ga= - \max_s \, \min_p  H^n_\ga(s,p) \Txt{for any $\ga$.}\]

By  adapting the argument of Corollary \ref{corunicumque}, we can finally prove:

\begin{Theorem}\label{tombola} Assume that $H^n_\ga(s,p)  \longrightarrow H_\ga(s,p)$ uniformly in $[0,1] \times [T_0,T_1)$, for any $\ga$, that  $g_n$ is uniformly convergent to  a function $g$ in $\Gamma$,  and that  $c^n_x \longrightarrow c_x$ for any $x$. Let us denote by $u_n$ the sequence of continuous  solutions to (HJ$\Gamma$) with $H^n_\ga$,  $c^n_x$ in place of $H_\ga$, $c_x$, respectively, and initial  datum $g_n$. Then the $u_n$  locally uniformly converge in $\Gamma \times [T_0,T_1)$ to the continuous solution $u$ of  (HJ$\Gamma$) with initial datum $g$ and flux limiter $c_x$.
\end{Theorem}
\begin{proof}
 We claim that the $u_n$'s are equibounded and equicontinuous.  We denote by $g_n^k$  sequences of Lipschitz continuous functions on $\Gamma$ with
\[ \lim_k g_n^k=g_n \txt{for any $n$, uniformly in $\Gamma$ \quad and} \; \Lip g_n^k=k,\]
see Lemma \ref{hopf} for the existence of such approximations. We denote by $u_n^k$ the solutions corresponding to $g_n^k$ and the flux limiter $c^n_x$. Since, in view of Remark \ref{postpasqua}, the Lipschitz constants of $u_n^k$ only depend on the Hamiltonians $H^n_\ga$, which uniformly converge to $H_\ga$, and the Lipschitz constants of $g_n^k$, we derive that the family of $u_n^k$, with $k$ fixed and $n$ varying in $\N$, is equiLipschitz continuous. We set
\[\ell_n  = \sup_k \Lip u_n^k.\]
Since the rate of convergence of $g^k_n$ to $g_n$ depends on the continuity modulus of $g_n$, see Lemma \ref{hopf} {\bf (ii)}, and the $g_n$'s  are equicontinuous, there exists an infinitesimal sequence $a_k >0$ with
\[a_k \geq |g_n^k-g_n|_\infty \Txt{for any $n$.}\]
This implies by Corollary \ref{corunicumque}
\[a_k \geq |u_n^k -u_n|_\infty \Txt{for any $n$.}\]
We derive that
\[|u_n(s_1,t_1) -u_n(s_2,t_2)| \leq  \inf_k \{a_k + \ell_k \, (|s_1-t_1|+ |s_2-t_2|)\} \Txt{for any $n$.}\]
Therefore the $u_n$'s are equicontinuous and, since they start from the initial data $g_n$ which are equibounded, they are locally equibounded as well,  which proves the claim. This in turn implies by Ascoli Theorem that the $u_n$ are locally uniformly convergent, up to subsequences, to a continuous function denoted by $u$.

Taking into account the usual stability results in viscosity solutions theory, to prove that $u$ is solution of (HJ$\Gamma$), it is enough to check that it satisfies  the condition {\bf (ii)} in the definition of supersolution.  This can be done arguing as in the last part of Corollary \ref{corunicumque}, with obvious modification.
\end{proof}

\bigskip

\begin{appendix}

\section{$t$--partial sup convolutions} \label{AA}

We set, as usual, $R = (a,b) \times (T_0,T_1)$, with $T_1 \leq + \infty$. Given an uniformly  continuous function $u$ in $\ov R$ with continuity modulus   $\om$,  we define its $t$--partial sup--convolutions
via:
\begin{equation}\label{supco}
  u^\de(s,t)= \max \left \{ u(s,r) - \frac
1{2\de} \, (r -t)^2  \mid  r \in \ov{[T_0,T_1)} \right \},
\end{equation}
for $\de >0$, note that the maximum in \eqref{supco} does exist, even if  $T_1=+\infty$, because $u$ has sublinear growth for $t$ going to $+ \infty$. The next proposition summarizes  some properties of interest of this regularization.

\begin{Proposition}\label{propsupco} We  have
\begin{itemize}
    \item[{\bf (i)}] $u^\de$ uniformly converges to $u$ in $R$ as $\de$ goes to $0$;
    \item[{\bf (ii)}] if $u$ is a subsolution of \eqref{HJloc} in $R$, then for any $\de$, there exists $T_\de=\OO(\sqrt \de)$ such that  $ u^\de$ is a Lipschitz continuous subsolution of \eqref{HJloc} in $(a,b) \times (T_0+T_\de, T_1)$.
\end{itemize}
\end{Proposition}
\begin{proof}  We denote by $\om$ a continuity modulus of $u$ in $R$, and consider two positive constants $a$, $\ell$  with
\[|u(s_i,t_1)- u(s_2,t_2)| \leq a + \ell \, (|s_1-s_2| + |t_1-t_2|) \txt{$(s_i,t_i)$, \,$i=1,2$ \, in $R$.}\]
Given $(s,t) \in R$, we say that $r$ is $u^\de$--optimal for $(s,t)$ if it realizes the maximum in \eqref{supco}. To  estimate $|r-t|$, we start from
\begin{eqnarray*}
  0 &\leq& u^\de(s,t) -u(s,t) = u(s,r)- u(s,t) - \frac 1{2 \de} \, (t-r)^2 \\ &\leq& a + \ell \, |t-r| - \frac 1{2 \de} \, (t-r)^2
\end{eqnarray*}
which implies
\[\frac  1{2 \de} \, (r-t)^2 \leq a + \ell \, |t-r|.\]
We deduce that there exists  $T_\de=\OO(\sqrt \de)$  such that
\begin{equation}\label{proppo1}
 |t-r| \leq T_\de \Txt{for any $(s,t)$, $r$ $u^\de$--optimal for $(s,t)$.}
\end{equation}
We know that $u^\de(s,\cdot)$ is semiconvex and consequently locally Lipschitz continuous, for any $s$. We derive from \eqref{proppo1} that it is  globally Lipschitz continuous in $\R^+$ and the Lipschitz constant is independent of $s$. This property depends on the fact that if $u(s,\cdot)$ is differentiable at $t$ then the $u^\de$--optimal point for $(s,t)$ is  unique and the derivative is given by $\frac {r-t}\de$. We therefore have that the Lipschitz constant of $u^\de(s,\cdot)$ in $[T_0,T_1)$ is  estimated from above by $\frac{T_\de}\de$.  We further derive
\begin{eqnarray*}
  u^\de(s_1,t_1) - u^\de(s_2,t_2)  &=& \big ( u^\de(s_1,t_1) - u^\de(s_1,t_2) \big )+ \big ( u^\de(s_1,t_2) - u^\de(s_2,t_2) \big )\\
  &\leq & \big ( u^\de(s_1,t_1) - u^\de(s_1,t_2) \big ) + \big ( u(s_1,r_1) - u(s_2,r_1) \big )
  \\  &\leq&  \frac{T_\de}\de \, |t_1-t_2| + \om (|s_1-s_2|),
\end{eqnarray*}
where $r_1$ is $u^\de$--optimal for $(s_1,t_1)$, and
\[u^\de(s,t) - u(s,t) \leq u(s,r)-u(s,t) \leq \om(T_\de) \Txt{for any $(s,t)$,}\]
which shows that $u^\de$ is uniformly continuous in $(s,t)$, and item {\bf (i)}. Taking into account that by \eqref{proppo1} $T_0$ cannot be $u^\de$--optimal for $(s,t)$ whenever $t > T_0 + T_\de$, we have  that if $\varphi$ is supertangent to $u^\de$ at a point $(s_0,t_0)$  with $t_0 \in (T_0+T_\de,T_1)$  and $r_0$ is $u^\de$--optimal  for $(s_0,t_0)$ then
\[(s,t) \mapsto \varphi(s,t+ (t_0-r_0))\]
is supertangent to $u$ at $(s_0,r_0)$, constrained to $\ov R$ if $r=T_1$. Therefore, if $u$ is subsolution of \eqref{HJloc} in $R$, then
\[ \varphi_t(s_0,t_0) + H(s_0, \varphi'(s_0,t_0) \leq 0\]
which  shows that $u^\de$ is subsolution of the same equation in $(a,b) \times (T_0+T_\de, T_1)$. Consequently, taking into account that the Hamiltonian is coercive and $u$ is continuous in $(s,t)$ plus Lipschitz continuous in $t$ with Lipschitz constant independent of $s$, $u^\de$ is Lipschitz continuous in $[a,b] \times [T+T_\de,+T_1)$. This shows item {\bf (ii)}, and concludes the proof.

\end{proof}

\bigskip

\section{Some proofs of results in Section \ref{interval}} \label{A}

We recall that $R= (a,b) \times (T_0,T_1)$.

\smallskip

\begin{proof}[{\bf Proof of Theorem \ref{nobarles}}] We consider the $t$--partial sup convolutions of $u$ denoted by $u^\de$. We exploit that $u^\de$ uniformly converges to $u$ in $R$ as $\de \to 0$, see Proposition \ref{propsupco}, and $T_\de \to 0$ as $\de \to 0$, see the statement of Proposition \ref{propsupco} for the definition of $T_\de$. Given an arbitrary  $\eps > 0$, we then have
\begin{equation}\label{nobarles1}
v + \eps > u^\de  \Txt{in $\partial_p R \cup [a,b] \times [T_0,T_0+T_\de]$, for $\de$  small.}
\end{equation}
Let $(s_0,t_0) \in R$, then $(s_0,t_0) \in (a,b) \times (T_0+T_\de,T_1)$ for  $\de$ sufficiently small. Since by Proposition \ref{propsupco} $u^\de$ is a Lipschitz continuous subsolution to \eqref{HJloc} in $(a,b) \times (T_0+T_\de,T_1)$, and by \eqref{nobarles1}
\[ v+ \eps  > u^\de \Txt{in $\partial_p \big ( (a,b) \times (T_0+T_\de,T_1) \big )$, for $\de$ small}\]
we derive from Theorem \ref{barles} that
\[v(s_0,t_0) + \eps \geq u^\de(s_0,t_0) \Txt{for $\de$ small.}\]
Passing at the limit for $\de \to 0$ we then have
\[v(s_0,t_0) + \eps \geq u(s_0,t_0).\]
This proves the assertion for $\eps$, $(s_0,t_0)$  have been arbitrarily chosen.
\end{proof}

\medskip
\begin{proof}[ {\bf Proof of Proposition \ref{hoje}}] Let $(s_0,t_0)$ a differentiability point of $u$ in $R$, then
\[u_t(s_0,t_0) + H(s_0,u'(s_0,t_0) \leq 0,\]
and  $|u_t(s_0,t_0)| \leq M$ by \eqref{hoje1}, consequently
\[ H(s_0,u'(s_0,t_0)) \leq M,\]
since by \eqref{assu} $u_t(s_0,t_0) <0$.
We deduce from the convex character of the Hamiltonian and from  \cite[Proposition 2.3.16]{C1}
\[H(s,p) \leq M \Txt{for any $(s,t) \in R$, $p \in \partial u(t,\cdot)$ at $s$.}\]
We therefore get  \eqref{hoje2} letting $t$ go to the boundary of $\ov{[T_0,T_1)}$ and exploiting  the stability properties of viscosity subsolutions.

\end{proof}

\medskip

 We need introducing some preliminary material before attacking the proof of Proposition \ref{maxistar}.

\smallskip

\begin{Lemma}\label{premaxi} Due to condition \eqref{assu},  every uniformly continuous subsolution of \eqref{HJloc} in $R$ is nonincreasing in $t$.
\end{Lemma}
\begin{proof} Let $u$ be a function as in the statement. We fix   $t_1 > t_2 \in (T_0,T_1)$. We know from  Proposition \ref{propsupco} {\bf (ii)}  that for $\de$ sufficiently small, the $u^\de$'s  are subsolutions of \eqref{HJloc} in $(a,b) \times (T,T_1)$ for some $T <t_2$. Since the Hamiltonian is convex and   $u^\de$ is Lipschitz continuous, we have
\[r + H(s,p) \leq 0 \Txt{for any $(s,t) \in (0,1) \times (T,+\infty)$, any $(p,r) \in \partial u^\de(s,t)$.}\]
We deduce from condition \eqref{assu} and  \cite[Proposition 2.3.16]{C1}
\begin{equation}\label{ciarea}
 u^\de_t(s,t) \leq 0 \Txt{at any point $(s,t)$ where $u$ is $t$--differentiable.}
\end{equation}
We now fix $s$, we derive from \eqref{ciarea} and the fact that $u^\de(s,\cdot)$ is Lipschitz continuous  that
\[u^\de(s,t_1) \geq u^\de(s,t_2)\]
and consequently taking into account that $u^\de$ uniformly converges  to $u$ in $R$
\[u(s,t_1) \geq u(s,t_2).\]
This concludes the proof since $s$, $t_1$, $t_2$ have been arbitrarily chosen
\end{proof}

\smallskip

\begin{Lemma}\label{stephan} Let $\A$ be a family of uniformly continuous functions from $R$ to $\R$ locally equibounded, with a  common continuity modulus $\om$, and closed in the local uniform topology, then
\[u(s,t)= \sup \{v(s,t) \mid v \in \A\}\]
is uniformly  continuous with continuity modulus  $\om$.
\end{Lemma}
\begin{proof} Given $(s,t) \in R$, we consider a sequence $v_n$ contained in $\A$ with
\[v_n(s,t) \to u(s,t).\]
By applying Ascoli theorem, we see that the $v_n$ locally uniformly converges, up to subsequences, to some $v \in \A$. We conclude that, given $(s,t) \in R$
\[u(s,t)= v(s,t) \Txt{for some $v \in \A$.}\]
Given $(s_i,t_i) \in R$, $i=1,2$, we denote by $v_i$ the functions of $\A$ satisfying the above property for $(s_i,t_i)$. We have
\[u(s_1,t_1) - u(s_2,t_2) \leq u_1(s_1,t_1) - u_1(s_2,t_2) \leq \om (|s_1-s_2| + |t_1-t_2|)\]
and
\[u(s_1,t_1) - u(s_2,t_2) \geq u_2(s_1,t_1) - u_2(s_2,t_2) \geq - \om  (|s_1-s_2| + |t_1-t_2|).\]
This shows the assertion.
\end{proof}

\medskip

\medskip
\begin{proof}[ {\bf Proof of Proposition \ref{maxistar}}]  Let $\widetilde u$ be a function in the set appearing in \eqref{maxistar1}, we denote by $\om$ its uniform continuity modulus.
We define
\[\widetilde {\mathcal S} =  \{u \;\hbox{un. cont. subsol of \eqref{HJloc} in  $R$ with cont. modulus $\om + M_0 r$},\, u \leq w_0 \;\hbox{on $\partial_p^-R$}\},\]
and
\[\widetilde v(s,t)= \sup \{u(s,t) \mid u \in \widetilde {\mathcal S}\}.\]
Since  all uniformly continuous subsolutions of \eqref{HJloc} are nonincreasing by Lemma \ref{premaxi}, and    the function $ w_0(s,T_0)-M_0 (t-T_0)$ belongs to $\widetilde{\mathcal S}$, $\widetilde v$  is not affected if we assume the $u$'s  of $\widetilde {\mathcal S}$ to satisfy in addition \[ w_0(s,T_0)  - M_0 \, (t-T_0) \leq u(s,t) \leq w_0(s, T_0)  \qquad\hbox{for $(s,t) \in R$}\]
Therefore the functions of $\widetilde {\mathcal S}$ have a common uniformity modulus and are locally equibounded. We derive  that $\widetilde v$ is uniformly continuous with continuity modulus $\om + M_0 r$ by Lemma \ref{stephan}, and subsolution to \eqref{HJloc} by basic  properties of viscosity solutions theory. We further have $\widetilde v(\cdot,T_0)=w_0(\cdot,T_0)$.

We fix  a time $h >0$ with $T_0+h <T_1$, and consider the family of functions $\overline{\mathcal S}$ defined as
\[  \{u(s,t) \;\hbox{un. cont. subsol of \eqref{HJloc} with cont. modulus $\om + M_0 r$}, \;  u \leq w_0-M_0 \, h \;  \hbox{on $\partial_p^-R$}\}, \]
it is clear that
\begin{equation}\label{faber111}
 \widetilde  v(s,t) - M_0\,h = \sup \{u(s,t) \mid u \in \overline{\mathcal S}\}.
\end{equation}
We  consider an   $u \in \overline{\mathcal S}$  coinciding with  $w_0-M_0 \, h$ in $[a,b] \times \{T_0\}$, and  define
\[\overline u(s,t) = \left \{ \begin{array}{cc}
   w_0(s,T_0) - M_0 \, (t -T_0) & \quad s \in [a,b],\, t \in [T_0,T_0+h) \\
    u(s,t-h) & \quad s \in [a,b],\,t \in [T_0+h, T_1)
  \end{array} \right . \]
  We have
  \begin{eqnarray*}
    \ov u(s,t_0) &=& w_0(s,t_0) \Txt{for $s \in [a,b]$} \\
    \ov u(a,t) &=& w_0(a,t) - M_0 \, (t-T_0) \leq w_0(a,t) \txt{for $t \in [T_0,T_0+h]$}\\
    \ov u(a,t) &=&  u(a,t-h) \leq  w_0(a,t-h) - M_0 \, h \leq w_0(a,t) \txt{for $t \in [T_0+h,T_1)$,}
  \end{eqnarray*}
 in addition $\overline u$ is subsolution of \eqref{HJloc}  by Proposition \ref{melaton}, and  has continuity modulus $\om + M_0 r$. Note that for this last property we have used that it is Lipschitz continuous with Lipschitz constant $M_0$ in $[a,b] \times [t_0,t_0+h]$.   We conclude that $\overline u$  belongs to $\widetilde {\mathcal S}$, so that
 \begin{equation}\label{faber333}
\widetilde v(s,t+h) \geq \overline u(s,t+h) = u(s,t) \Txt{ for any $s \in [a,b]$, $t \in  [T_0,T_1-h)$.}
 \end{equation}
Since  $u$ has been arbitrarily taken in $\overline{\mathcal S}$, the prescription of the value on $[a,b] \times \{t_0\}$ being  not a real restriction,  we  derive  from \eqref{faber111},  \eqref{faber333} that
\[\widetilde v(s,t+h) \geq \widetilde v(s,t) -  M_0 \, h    \Txt{in $[a,b] \times [T_0,T_1-h)$.}\]
  Taking into account that $\widetilde v$ is nonincreasing in $t$, we finally get
  \[|\widetilde v(s,t+h) - \widetilde  v(s,t)| \leq M_0 \, h \Txt{for any $s$, $t$,}\]
which implies, $h$ being arbitrary, that $\widetilde v(s,\cdot)$ is Lipschitz continuous with Lipschitz constant $M_0$ for any $s$.
  Taking into account  the coercivity of $H$ and that $\widetilde v$ is subsolution to \eqref{HJloc}, we derive that $(\widetilde v', \widetilde v_t)$ is bounded in the viscosity sense in $R$ so that $\widetilde v$  is Lipschitz continuous  and the estimates \eqref{trump1bis}, \eqref{trump2bis} holds true with $\widetilde v$ in place of $v$. We define
\[ \mathcal S=  \{u(s,t) \mid u \;\hbox{Lip.  subsol. of \eqref{HJloc} in  $R$ satisfying \eqref{trump1bis}, \eqref{trump2bis}},\, \, u \leq w_0 \;\hbox{in $\partial_p^- R$}\}\]
Since $\widetilde u$ has been arbitrarily taken in the family of functions defining $v$, $\widetilde u \leq \widetilde v$ and $\widetilde v \in \mathcal S$, we have that
\[v(s,t) = \sup\{u(s,t) \mid u \in \mathcal S\}.\]
Arguing as in the first part of the proof, we see that $v$ is subsolution of \eqref{HJloc},   that it is Lipschitz continuous satisfying \eqref{trump1bis}, \eqref{trump2bis} and  coincides with $w_0$ on $[a,b] \times \{T_0\}$.
Finally, by exploiting its maximality, we show, via Perron--Ishii method, that $v$  is solution to \eqref{HJloc}.
\end{proof}

\medskip

\begin{proof}[\bf{Proof of Lemma \ref{hopf}}] \,
  We denote by $\om$ a continuity modulus for $u$ in $C$ that can be taken in the form
\[\om(r)= \inf_k \{a_k + \ell_k \, r\} \Txt{with  $a_k$, $\ell_k$ positive, $a_k \to 0$, for $r\geq 0$}.\]
We fix $k_0 \in \N$ and set $a=a_{k_0}$, $\ell= \ell_{k_0}$, so that
      \[|u(s_1,t_1)-u(s_2,t_2)| \leq a + \ell \, (|s_1-s_2| + |t_1-t_2|).\]
This implies that there  are maximizers/minimizers  in the formulae yielding $u^{[n]}$, $u_{[n]}$ for $n > \ell$. Let $(s_1,t_1)$, $(s_2,t_2)$ with $u^{[n]}(s_1,t_1) \geq u^{[n]}(s_2,t_2)$,  $n > \ell$, we denote by $(z_1,t_1)$ a maximizer for $u^{[n]}(s_1,t_1)$. We have
      \[u^{[n]}(s_1,t_1) -u^{[n]}(s_2,t_2) \leq n \, (|s_1-z_1| + |t_1-r_1|- |s_2-z_1|- |t_2-r_1| )\]
which shows item {\bf (i)} for $u^{[n]}$, the same argument applies to $u_{[n]}$.  Given $(s,t)$, if $(z_n,r_n)$ is a point realizing the maximum  for $u^{[n]}(s,t)$, $n > \ell$,  then
\[ |s-z_n| + |t-r_n| \leq \frac 1n \, (u(z_n,r_n) - u(s,t)) \leq \frac an + \frac \ell n \, (|s-z_n| + |t-r_n|)\]
which implies
\[|s-z_n| + |t-r_n| =  \OO(1/n).\]
We then have
\[u^{[n]}(s,t) - u(s,t) \leq u(z_n,r_n) - u(s,t) \leq \om(\OO(1/n))\]
and consequently $u^{[n]}$ uniformly converge to $u$ in $\ov Q$. The same property  holds for  $u_{[n]}$. We see, in addition, that $|u^{[n]}-u|_\infty$ ($|u_{[n]}-u|_\infty$) solely depends on the continuity modulus of $u$.
\end{proof}

\medskip
\begin{proof}[\bf{Proof of Lemma \ref{catania1}}] \, We denote by $u$, $v$ two Lipschitz continuous subsolutions and by $w$ their minimum. Let $(s_0,t_0)$ be a differentiability point of $w$ and assume that $w(s_0,t_0) = u(s_0,t_0)$, then
\[(w_t(s_0,t_0), Dw(s_0,t_0)) \in D^- u(s_0,t_0) \subset \partial u(s_0,t_0),\]
where $D^-$ denotes the viscosity subdifferential,  and by the convexity of $H$
\[w_t(s_0,t_0) + H(s_0,D w(s_0,t_0)) \leq 0.\]
This shows that $w$ is a.e. subsolution which is equivalent, thanks  again to the convexity of $H$, of being a viscosity subsolution.
\end{proof}

 \bigskip

\color{black}

\bigskip

%%%%%%%%%%%%% BIBLIOGRAPHY %%%%%

\end{appendix}

\vspace{10 pt}

\begin{thebibliography}{10}
\expandafter\ifx\csname natexlab\endcsname\relax\def\natexlab#1{#1}\fi
\expandafter\ifx\csname bibnamefont\endcsname\relax
  \def\bibnamefont#1{#1}\fi
\expandafter\ifx\csname bibfnamefont\endcsname\relax
  \def\bibfnamefont#1{#1}\fi
\expandafter\ifx\csname citenamefont\endcsname\relax
  \def\citenamefont#1{#1}\fi
\expandafter\ifx\csname url\endcsname\relax
  \def\url#1{\texttt{#1}}\fi
\expandafter\ifx\csname urlprefix\endcsname\relax\def\urlprefix{URL }\fi
\providecommand{\bibinfo}[2]{#2}
\providecommand{\eprint}[2][]{\url{#2}}



\bibitem{BI} Guy Barles, Ariela Briani, Emmanuel Chasseigne,  and Cyril Imbert. \newblock  Flux-limited and classical viscosity solutions for regional control problems.
    \newblock {\em ESAIM Control Optim. Calc.
Var.} 24 :  1881–1906, 2018.

\bibitem{C1} Frank H. Clarke.
\newblock Optimization and nonsmooth analysis.
\newblock {\em Classics in Applied Mathematics} Wiley, New York,
1983.

\bibitem{GIM} Giulio Galise, Cyril  Imbert, and R\'{e}gis Monneau.
\newblock  A junction condition by specified homogenization and application to traffic lights.
\newblock {\em Analysis and PDE} 8 : 1891–1929, 2015.

\bibitem{IM} Cyril Imbert and R\'{e}gis Monneau.
\newblock Flux-limited solutions for quasi-convex Hamilton--Jacobi equations on networks.
\newblock {\em Ann. Sci. Ec. Norm. Sup\'{e}r.} 50:  357–448, 2017.


\bibitem{IMZ} Cyril Imbert, R\'{e}gis Monneau, and Hasnaa Zidani.
\newblock  A Hamilton-Jacobi approach to junction problems and application to traffic flows.
\newblock{\em ESAIM Control Optim. Calc. Var.} 19:
 pp. 129–166, 2013.

\bibitem{PoSi} Marco Pozza and Antonio Siconolfi.
\newblock Discounted Hamilton-Jacobi Equations on Networks and Asymptotic Analysis.
\newblock {\em Indiana Un. Math J.}  to appear.

\bibitem{SC}
  Dirk Schieborn and Fabio Camilli.
\newblock Viscosity solutions of Eikonal equations on topological networks.
\newblock {\em  Calc. Var. Partial Differential Equations} 46: 671-686, 2013.

\bibitem{SiconolfiSorrentino}
Antonio Siconolfi and Alfonso Sorrentino.
\newblock Global Results for Eikonal Hamilton-Jacobi Equations on Networks
\newblock {\em Analysis and PDE (11)} 1: 171--211, 2018.

\bibitem{Sunada}
Toshikazu Sunada.
\newblock Topological crystallography. With a view towards discrete geometric analysis.
\newblock {\em Surveys and Tutorials in the Applied Mathematical Sciences}, 6. Springer, Tokyo, xii+229 pp., 2013.

\end{thebibliography}
\end{document}